\documentclass[12pt,a4paper]{amsart}
\usepackage{mathtools,amssymb,amsmath,amsopn,amsthm,graphicx,MnSymbol,centernot,stmaryrd,array}
\usepackage{tikz-cd}
\usepackage{enumitem}
\usepackage{amsaddr}
\usepackage{etoolbox}
\patchcmd{\subsection}{-.5em}{.5em}{}{}
\usetikzlibrary{matrix}
\begin{document}

\newtheorem{definition}{Definition}[section]
\newtheorem{definitions}[definition]{Definitions}
\newtheorem{deflem}[definition]{Definition and Lemma}
\newtheorem{lemma}[definition]{Lemma}
\newtheorem{proposition}[definition]{Proposition}
\newtheorem{theorem}[definition]{Theorem}
\newtheorem{corollary}[definition]{Corollary}
\newtheorem{algo}[definition]{Algorithm}
\theoremstyle{remark}
\newtheorem{rmk}[definition]{Remark}
\theoremstyle{remark}
\newtheorem{remarks}[definition]{Remarks}
\theoremstyle{remark}
\newtheorem{notation}[definition]{Notation}
\theoremstyle{remark}
\newtheorem{example}[definition]{Example}
\theoremstyle{remark}
\newtheorem{examples}[definition]{Examples}
\theoremstyle{remark}
\newtheorem{dgram}[definition]{Diagram}
\theoremstyle{remark}
\newtheorem{fact}[definition]{Fact}
\theoremstyle{remark}
\newtheorem{illust}[definition]{Illustration}
\theoremstyle{remark}
\newtheorem{que}[definition]{Question}
\theoremstyle{definition}
\newtheorem{conj}[definition]{Conjecture}
\newtheorem{scho}[definition]{Scholium}
\newtheorem{por}[definition]{Porism}
\DeclarePairedDelimiter\floor{\lfloor}{\rfloor}

\renewenvironment{proof}{\noindent {\bf{Proof.}}}{\hspace*{3mm}{$\Box$}{\vspace{9pt}}}
\author[Sardar, Kuber]{Shantanu Sardar and Amit Kuber}
\address{Department of Mathematics and Statistics\\Indian Institute of Technology, Kanpur\\ Uttar Pradesh, India}
\email{shantanusardar17@gmail.com, askuber@iitk.ac.in}
\title{{Variations of the bridge quiver for domestic string algebras}}

\begin{abstract}
In the computation of some representation-theoretic numerical invariants of domestic string algebras, a finite combinatorial gadget introduced by Schr\"{o}er--the \emph{bridge quiver} whose vertices are (representatives of cyclic permutations of) bands and whose arrows are certain band-free strings--has been used extensively. There is a natural but ill-behaved partial binary operation, $\circ$, on the larger set of \emph{weak bridges} such that bridges are precisely the $\circ$-irreducibles. With the goal of computing the order type of any hammock in a later work we equip an even larger set of \emph{weak arch bridges} with another partial binary operation, $\circ_H$, to obtain a finite category. Each weak arch bridge admits a unique $\circ_H$-factorization into \emph{arch bridges}, i.e., the $\circ_H$-irreducibles.
\end{abstract}

\maketitle

\newcommand\B{\mathfrak{B}}
\newcommand\Pp{\mathcal{P}}
\newcommand\hh{\mathfrak{h}}
\newcommand\Red[1]{\mathrm{R}_{#1}}
\newcommand\HRed[1]{\mathrm{HR}_{#1}}
\newcommand\M{\mathcal{M}}
\newcommand\Q{\mathcal{Q}}
\newcommand\T{\mathcal{T}}
\newcommand{\bb}{\mathfrak b}
\newcommand{\ch}{\circ_H}
\newcommand{\bua}[1]{\mathfrak b^{\alpha}(#1)}
\newcommand{\bub}[1]{\mathfrak b^{\beta}(#1)}
\newcommand{\bla}[1]{\mathfrak b_{\alpha}(#1)}
\newcommand{\blb}[1]{\mathfrak b_{\beta}(#1)}
\newcommand{\uu}{\mathfrak u}
\newcommand{\vv}{\mathfrak v}
\newcommand{\ww}{\mathfrak w}
\newcommand{\xx}{\mathfrak x}
\newcommand{\yy}{\mathfrak y}
\newcommand{\zz}{\mathfrak z}
\newcommand{\MM}{\mathfrak M}
\newcommand{\mm}{\mathfrak m}
\newcommand{\sbq}{\mathfrak s}
\newcommand{\tbq}{\mathfrak t}
\newcommand{\sk}[1]{\{#1\}}
\newcommand{\STR}[1]{\mathrm{Str}(#1)}
\newcommand{\Ba}[1]{\mathrm{Ba}(#1)}
\newcommand{\HQ}{\mathcal{HQ}^\mathrm{Ba}}
\newcommand{\bHQ}{\overline{\mathcal{HQ}}^\mathrm{Ba}}

\section{Introduction}
In the representation theory of finite dimensional algebras over an algebraically closed field $\mathcal K$, the class of string algebras form a `test subclass' for various conjectures regarding the class of tame representation type algebras due to the possibility of explicit computations. String algebras are presented as a certain quotient $\mathcal KQ/\langle\rho\rangle$ of the path algebra $\mathcal KQ$ of a quiver $Q$, where $\rho$ is a set of monomial relations. A complete classification of the indecomposable finite length modules over a string algebra into string and band modules is essentially due to Gelfand and Ponomarev \cite{GP}, where strings and bands are certain walks in the quiver.

The complexity of the set of strings is the major source of the complexity in the representation-theoretic study of string algebras. In his doctoral thesis \cite[\S~4]{SchroerThesis}, Schr\"{o}er introduced a finite combinatorial gadget called the \emph{bridge quiver} for the study of strings and bands in the context of domestic string algebras--this subclass is characterized by the existence of only finitely many bands--and showed \cite[\S~4.9]{SchroerThesis} that each string for a domestic string algebra can be generated by a path in the bridge quiver. Laking \cite[\S~2.5]{LakingThesis} used slightly different definition of the bridge quiver, whereas in \cite{GKS}, Gupta and the authors introduced the concept of a prime band for non-domestic string algebras, and then extended the concept of a finite (prime) (weak) bridge quiver to all string algebras; again the definition is slightly different. It was also shown that each string is ``generated by'', in a precise sense, a path in the extended (prime) bridge quiver.

Graph-theoretic properties as well as numerical invariants associated with the bridge quiver are key tools in the study of string algebras. Suppose $n$ denotes the maximal length of a path in the bridge quiver. For a domestic string algebra $\Lambda$ Schr\"{o}er \cite[Theorems~3.10,4.3]{Sch} showed that $\mathrm{rad}_\Lambda^{\omega\cdotp(n+1)}\neq0$ but $\mathrm{rad}_\Lambda^{\omega\cdotp(n+2)}=0$, where $\mathrm{rad}_\Lambda$ denotes the radical of the category $\mathrm{mod}\mbox{-}\Lambda$ of finite dimensional
right $\Lambda$-modules. Laking, Prest and Puninski \cite{LPP} showed that $\mathrm{KG}(\Lambda)=\mathrm{CB}(\mathrm{Zg}_\Lambda)=n+2$, where $\mathrm{KG}(\Lambda)$ stands for the Krull-Gabriel dimension of $\Lambda$ whereas $\mathrm{CB}(\mathrm{Zg}_\Lambda)$ stands for the Cantor-Bendixson rank of the Ziegler spectrum of $\Lambda$. Gupta and the authors \cite{GKS} defined the subclass of meta-torsion-free algebras of the class of non-domestic string algebras using directed cycles in the (prime) bridge quiver and showed that for such algebra $\Lambda$, $\omega\leq\min\{\lambda\mid\mathrm{rad}_\Lambda^\lambda=\mathrm{rad}_\Lambda^{\lambda+1}\}\leq\omega+2$.

For a domestic string algebra, Schr\"{o}er used the bridge quiver to investigate (different variations of) \emph{hammock (po)sets} consisting of string modules and band modules. The simplest version of hammock, denoted $H_l(v)$, consists of strings starting at the vertex $v$ of the quiver $Q$, and can be naturally equipped with a total order, say $<_l$, so that for $\xx,\yy\in H_l(v)$ if $\xx<_l\yy$ then there is a canonical map, known as a \emph{graph map}, $M(\xx)\to M(\yy)$ between the corresponding string modules. Krause \cite{K} and Crawley-Boevey \cite{CB} showed that slightly more general graph maps form a basis of the finite dimensional vector space $\mathrm{Hom}_\Lambda(M(\xx),M(\yy))$. Schr\"{o}er \cite[\S~2.5]{SchroerThesis} showed that $(H_l(v),<_l)$ is a bounded discrete linear order. Further Prest and Schr\"{o}er showed (essentially \cite[Theorem~1.3]{PS}) that it has finite dimension--this notion of dimension is known as the Hausdorff rank in order-theoretic literature.

During personal communication Prest said that a better version of the bridge quiver would lead to a better understanding of the module category. Exact computation of the order type of the hammocks $(H_l(v),<_l)$ is the goal of the sequel \cite{SK}. It would shed more light on the canonical factorization of graph maps, and in turn, on the structure of the category $\mathrm{mod}\mbox{-}\Lambda$, and the purpose of the current paper is to lay down its combinatorial foundation. We realised that we needed a subquiver of the weak bridge quiver that contains the bridge quiver. After various unsuccessful attempts we finally arrived at the notion of the \emph{arch bridge quiver} that explicitly relies on the set $\rho$ of monomial relations. This is the major contribution of this paper. One unsuccessful attempt, namely that of the \emph{semi-bridge quiver}, that has a more natural description turned out to be a subquiver of the arch bridge quiver.

To explain the concept of arch bridges, we first look at the relationship between bridges and weak bridges. Weak bridges are band-free strings, and there are only finitely many band-free strings in any string algebra \cite[Proposition~3.1.7]{GKS}. We can equip the set of weak bridges with a partial composition, $\circ$, so that bridges are precisely the $\circ$-irreducible weak bridges. However iterated $\circ$-compositions of weak bridges need not be defined (Example \ref{nonassociativityweakbridge}), and thus $\circ$ is not associative. Moreover there could be multiple $\circ$-factorizations of the same weak bridge (Example \ref{multiplefactorizationweakbridge}).

The finite set of band-free strings is extended to the finite set (Corollary \ref{hereditaryHstringfinitude}) of \emph{H-reduced strings}. By introducing another composition operation on the set of weak bridges, which we call H-composition and denote by $\ch$, that has the property that all iterated compositions exist, we define a \emph{weak arch bridge} to be such an iterated $\ch$-composition. We show that a weak arch bridge is an H-reduced string (Theorem \ref{weakarchbridgeHreduced}) and that the H-composition of weak arch bridges is associative (Theorem \ref{Hassociativity}) thus making the weak arch bridge quiver a finite category with respect to $\ch$. An \emph{arch bridge} is then defined to be a $\ch$-irreducible weak arch bridge, and we show that every weak arch bridge factors uniquely as an H-composition of arch bridges (Theorem \ref{hisinjective}).

Although Ringel, Schr\"{o}er and other authors have considered only the hammocks $H_l(v)$, where $v$ is a vertex of the quiver, for applications in \cite{SK} we need to study slightly more general hammocks, $H_l(\xx_0)$, consisting of strings having $\xx_0$ as a left substring, where $\xx_0$ is an arbitrary string. On this note, recall that a (weak) (arch) bridge is an arrow between bands. However to study $H_l(\xx_0)$ we need to extend the (weak) (arch) bridge quiver by including arrows from the string $\xx_0$ to other length $0$ strings or bands as well as from a band to other length $0$ strings--the extension is accordingly called the \emph{extended (weak) (arch) bridge quiver}. All the results proved for (weak) (arch) bridges and H-compositions are carefully generalized to the extended setting. In fact, we show that essentially all H-reduced strings in $H_l(\xx_0)$ are arrows in the extended weak arch bridge quiver (Proposition \ref{extHreducedweakarchbr}).

One final contribution of the paper is the set of (counter)examples produced in this paper for various reasons. We believe that these algebras play an important role in the study of domestic string algebras beyond their use in \cite{SK}.

Due to the complex nature of the set of strings there is a lot of new terminology and notations in the paper. The paper is arranged as follows.

Some basic notation about hammocks and skeletal strings is introduced in \S\ref{secfund}. After recalling the definition of the (weak) bridge quiver of a domestic string algebra from \cite{GKS} in \S\ref{extbridge}, we introduce a na\"{i}ve notion of partial composition, $\circ$, of weak bridges and give examples to show its lack of desirable properties like associativity and unique factorization. The set of weak bridges is then partitioned into two classes (Definition \ref{nabndef}), which we call normal and abnormal, and study the latter class in \S\ref{abnarrows}, where we show that this class enjoys the above mentioned desirable properties (Corollary \ref{abnwbcancel} and Remark \ref{associativenaafact}). At the end of this section (Remark \ref{casebycaseanalysis}), we describe the structure of the composition of normal and abnormal arrows with examples.

The H-composition of weak bridges is introduced in \S\ref{secweakarchbrquiv} to describe the weak arch bridge quiver. In that section we also introduce the notion of an H-reduced string, and show that there are only finitely many of them in a domestic string algebra (Corollary \ref{hereditaryHstringfinitude}). The next section, \S\ref{secarchbrquiv}, forms the heart of the paper where we show that the H-composition of weak arch bridges enjoys desirable properties.

The latter part of \S\ref{extbridge} deals with the extended (weak) bridge quiver while \S\ref{secextarchbrquiv} extends the results of \S\ref{secarchbrquiv} to the extended (weak) arch bridge quiver. In the latter section the definitions are carefully laid down and the proofs are omitted.

Finally in \S\ref{newextsemibrquiv} we describe a subquiver of the extended arch bridge quiver, namely the extended semi-bridge quiver, that contains the extended bridge quiver and has a succinct definition. Example \ref{ex:archnotsemi} shows that there are arch bridges that are not semi-bridges.

\subsection*{Acknowledgements}
The first author thanks the \emph{Council of Scientific and Industrial Research (CSIR)} India - Research Grant No. 09/092(0951)/2016-EMR-I for the financial support. Both authors thank Bhargav Kale and Vinit Sinha for corrections in the earlier draft of the manuscript.

\subsection*{Data Availability Statement}
Data sharing not applicable to this article as no datasets were generated or analysed during the current study.

\section{Fundamentals of string algebras}\label{secfund}
Fix an algebraically closed field $\mathcal K$. In this paper we study domestic string algebras presented as $\Lambda:=\mathcal K\Q/\langle\rho\rangle$--here $\Q = (Q_0, Q_1, s, t)$ denotes a quiver, where $Q_0$ is the set of vertices, $Q_1$ is the set of arrows and $s,t:Q_1\to Q_0$ denote the source and target functions, and $\rho$ denotes a set of monomial relations. We assume that no two paths in $\rho$ are comparable, and set $\rho^{-1}:=\{\xx^{-1}\mid\xx\in\rho\}$. The term `string algebra' will always mean a quiver with relations together with a choice of $\sigma$ and $\varepsilon$ maps. The reader should refer to \cite[\S 2.1]{GKS} for the definitions of string algebra, (reduced) word, finite string, left/right $\mathbb N$- string and bands, as well as for the notations and conventions.

Recall that a string algebra is \emph{domestic} if it has only finitely many bands.

For a string $\xx$ with positive length, associate to it its \emph{sign}, denoted $\theta(\xx)$, by $\theta(\xx):=1$ (resp. $\theta(\xx):=-1$) if the first syllable of $\xx$ is an inverse (resp. direct) syllable. For a string $\uu$ with $|\uu|>0$ we set $\delta(\uu):=\begin{cases}\theta(\uu)&\mbox{if }\theta(\alpha_1)=\theta(\alpha_2)\mbox{ for all } \alpha_1,\alpha_2\in\uu;\\0&\mbox{otherwise}\end{cases}.$

Let $H_l(\xx_0)$ denote the left hammock set of the string $\xx_0$, i.e., the set of all strings of the form $\yy\xx_0$ where $\yy$ is any string, possibly of length $0$. The set $H_l(\xx_0)$ can be equipped with a total order $<_l$ described by Schr\"{o}er in \cite[\S~2]{SchroerThesis}. The description the left hammock set as a bounded discrete linear order as given by Lemma in \cite[\S2.5]{SchroerThesis} is also true for the total order $(H_l(\xx_0),<_l)$. An appropriate modification of \cite[Lemma~3.3.4]{GKS} shows that each string in $H_l(\xx_0)$ is ``generated by'', in a precise sense, by a not-necessarily-unique path in the extended bridge quiver $\Q^{\mathrm{Ba}}(\xx_0)$.

For $i\in\{+,-\}$, let $H_l^i(\xx_0):=\{\xx_0\}\cup\{\yy\xx_0\in H_l(\xx_0)\mid\theta(\yy)=i\}$. Then $$H_l(\xx_0)=H_l^+(\xx_0)\cup H_l^-(\xx_0),\quad\{\xx_0\}=H_l^+(\xx_0)\cap H_l^-(\xx_0).$$

For any string $\yy$ with $|\yy|>0$, let $\zz_l(\yy)$ (resp. $\zz_r(\yy)$) denote the longest left (resp. right) substring of $\yy$ such that $\delta(\zz_l(\yy))\neq0$ (resp. $\delta(\zz_r(\yy))\neq0$). If $|\yy|=0$, define $\zz_l(\yy):=1_{(s(\yy),-\sigma(\yy))}$ and $\zz_r(\yy):=1_{(t(\yy),-\epsilon(\yy))}$.

For a string $\xx_0$, $\yy\in H_l(\xx_0)$ and a band $\bb$, let $N(\xx_0;\bb,\yy)$ denote the maximum $n$ such that $\yy=\yy_2{\bb'}^n\yy_1\xx_0$ for some strings $\yy_1,\yy_2$ and a cyclic permutation $\bb'$ of $\bb$. Let $\B(\xx_0;\yy):=\{\bb\mid\bb\mbox{ is a band},\ N(\xx_0;\bb,\yy)>0\}$. Also define $N(\bb,\yy):=N(1_{(s(\yy),-\sigma(\yy))};\bb,\yy)$ and $\B(\yy):=\B(1_{(s(\yy),-\sigma(\yy))};\yy)$.

\begin{definition}
Say that a string $\yy$ is \emph{skeletal} if $N(\bb,\yy)\leq1$ for each band $\bb$. 
\end{definition}

The following follows easily from \cite[Propositions~3.1.7,3.4.2, Lemma~3.3.4]{GKS}.
\begin{proposition}\label{skeletalfinitude}
In a domestic string algebra there are only finitely many skeletal strings.
\end{proposition}

For a band $\bb$, say that a string $\xx$ is a \emph{$1$-step $\bb$-reduction} of a string $\yy$ if there is a partition $\yy=\yy_2\bb'\yy_1$, where $\bb'$ is a cyclic permutation of $\bb$, such that $\xx=\yy_2\yy_1$. Moreover say that a string $\xx$ is a \emph{reduction} (resp. \emph{$\bb$-reduction}) of $\yy$ if there is a finite sequence $\yy=\yy_0,\yy_1,\hdots,\yy_k=\xx$ of strings such that, for each $0<j\leq k$, $\yy_j$ is a $1$-step reduction (resp. $1$-step $\bb$-reduction) of $\yy_{j-1}$. When $\xx$ is a reduction of $\yy$ we say that $\yy$ is an \emph{extension} of $\xx$.
\begin{rmk}
Any reduction of a cyclic string is again cyclic.
\end{rmk}

If $\yy$ is a string, $\bb\in\B(\yy)$, and a $1$-step $\bb$-reduction of $\yy$ exists then we denote it by $\Red\bb(\yy)$. The \emph{skeleton} of $\yy$ is the string $\sk\yy:=\prod_{\bb\in\B(\yy)}\Red\bb^{N(\bb,\yy)-1}(\yy)$.

In the paper $\sqcup$ will denote disjoint union. For any definition, terminology or notation not explained here the reader should refer to \cite{GKS}.

\section{The extended weak bridge quiver}\label{extbridge}
The theory of (prime) bridge quivers for all string algebras has already been developed in detail in \cite[\S 3]{GKS} for (not necessarily domestic) string algebras, and is motivated by the concepts with similar names from \cite[\S 4.5,4.7,4.9]{SchroerThesis}. Below we recall some notations and definitions in the context of domestic string algebras.

Recall that $\Ba\Lambda$ is the collection of bands up to cyclic permutation and $Q_0^{\mathrm{Ba}}$ is a fixed set of representatives of elements in $\Ba\Lambda$. For $\bb_1,\bb_2\in Q_0^{\mathrm{Ba}}$, say that a finite string $\uu$ is a \emph{weak bridge} from $\bb_1\to\bb_2$ if it is band-free and the word $\bb_2\uu\bb_1$ is a string. Further say that a weak bridge $\bb_1\xrightarrow{\uu}\bb_2$ is a bridge if there is no band $\bb$ and weak bridges $\bb_1\xrightarrow{\uu_1}\bb$ and $\bb\xrightarrow{\uu_2}\bb_2$ such that one of the following holds.
\begin{itemize}
    \item $\uu=\uu_2\uu_1,|\uu_1|>0,|\uu_2|>0$.
    \item $\uu=\uu'_2\uu'_1,|\uu'_1|>0,|\uu'_2|>0,\uu_2=\uu'_2\uu''_2,\uu_1=\uu''_1\uu'_1$ and $\bb=\uu''_1\uu''_2$.
\end{itemize}

There is a trivial bridge $\bb\xrightarrow{1_{(t(\bb),\varepsilon(\bb))}}\bb$ for each band $\bb$ but these are not interesting, so the word (weak) bridge will be used only to refer to non-trivial (weak) bridge unless otherwise specified.
\begin{rmk}\label{bridgehasposlen}
If $\bb_1\xrightarrow{\uu}\bb_2$ is a bridge between bands in a domestic string algebra then $|\uu|\neq0$, for otherwise it is readily verified that both $\bb_1\bb_2$ and $\bb_2\bb_1$ are well-defined strings implying that $\Lambda$ is non-domestic.
\end{rmk}

By $Q_1^{\mathrm{Ba}}$ (resp. $\overline{Q}_1^{\mathrm{Ba}}$) we denote the set of all bridges (resp. weak bridges) between bands in $Q_0^{\mathrm{Ba}}$; this together with $Q_0^{\mathrm{Ba}}$ constitutes a quiver $\Q^{\mathrm{Ba}}=(Q_0^{\mathrm{Ba}},Q_1^{\mathrm{Ba}})$ (resp. $\overline{\Q}^{\mathrm{Ba}}=(Q_0^{\mathrm{Ba}},\overline{Q}_1^{\mathrm{Ba}})$) known as the \emph{bridge quiver} (resp. \emph{weak bridge quiver}). The bridge quiver of a domestic string algebra is finite and acyclic (cf. \cite[Theorem~3.1.6, Remark~3.2.4, Proposition~3.4.2]{GKS}, \cite[\S 4.4]{SchroerThesis}), and using similar arguments we can also show this result for the weak bridge quiver. The (weak) bridge quiver is equipped with the source and target functions $\sbq$ and $\tbq$ respectively.

The set $\overline{Q}_1^{\mathrm{Ba}}$ can be equipped with a partial binary operation $\circ$. Let $\uu,\uu'$ be weak bridges such that $\tbq(\uu)=\sbq(\uu')$. If $\sbq(\uu)\xrightarrow{\uu'\uu}\tbq(\uu')$ is a weak bridge then define $\uu'\circ\uu:=\uu'\uu$. On the other hand, if $N(\tbq(\uu),\uu'\uu)=1$ but $\sbq(\uu)\xrightarrow{\Red{\tbq(\uu)}(\uu'\uu)}\tbq(\uu')$ is a weak bridge then define $\uu'\circ\uu:=\Red{\tbq(\uu)}(\uu'\uu)$. The composition $\uu'\circ\uu$ does not exist if one of the following is true.
\begin{itemize}
    \item $\tbq(\uu)\neq\sbq(\uu')$;
    \item $\tbq(\uu)=\sbq(\uu')$, $N(\tbq(\uu),\uu'\uu)=1$ but $\Red{\tbq(\uu)}(\uu'\uu)$ does not exist;
    \item $\tbq(\uu)=\sbq(\uu')$, $N(\tbq(\uu),\uu'\uu)=1$ and $\Red{\tbq(\uu)}(\uu'\uu)$ exists but is not band-free.
\end{itemize}

\begin{rmk}\label{bridgesareirreducible}
A weak bridge $\uu$ is a bridge if and only if it cannot be written as $\uu=\uu_2\circ\uu_1$ for any non-trivial weak bridges $\uu_1,\uu_2$. In other words, bridges are precisely the $\circ$-irreducibles amongst $\overline{Q}_1^{\mathrm{Ba}}$.
\end{rmk}

The composition $\circ$ is not well-behaved as the following two examples demonstrate.
\begin{example}\label{multiplefactorizationweakbridge}
Consider the algebra $\Lambda$ from Figure \ref{Notuniquefact}. Here the bands are $\bb_1:=aB$, $\bb_2:=eD$, $\bb_3:=hG$ and $\bb_4=jK$ along with their inverses. If $\uu_{ij}$ denotes the weak bridge $\bb_i\xrightarrow{\uu_{ij}}\bb_j$ for $1\leq i<j\leq 4,\ (i,j)\neq(2,3)$, where $\uu_{14}=jiFc$, $\uu_{12}=ec$, $\uu_{13}=Fc$, $\uu_{24}= jiFD$ and $\uu_{34}=ji$ then  $\uu_{14}=\uu_{24}\circ\uu_{12}=\uu_{34}\circ\uu_{13}$ as shown in Figure \ref{exweakbrquivLambda}. This example shows that the factorisation of a weak bridge as a composition of bridges is not necessarily unique.

\begin{figure}[h]
\begin{minipage}[b]{0.55\linewidth}
\centering
\begin{tikzcd}
                                                                &                     & v_4                                                             &                                                                 &                                                                 &     \\
v_1 \arrow[r, "b"', bend right=35] \arrow[r, "a", bend left=35] & v_2 \arrow[r, "c"'] & v_3 \arrow[u, "e"', bend right=40] \arrow[u, "d", bend left=40] & v_5 \arrow[l, "f"] \arrow[r, "i"']                              & v_7 \arrow[r, "k", bend left=35] \arrow[r, "j"', bend right=35] & v_8 \\
                                                                &                     &                                                                 & v_6 \arrow[u, "g"', bend right=40] \arrow[u, "h", bend left=40] &                                                                 &    
\end{tikzcd}
\caption{$\Lambda$ with $\rho=\{cb,dc,ef,fh,ig,ki,dfg\}$}
\label{Notuniquefact}
\end{minipage}
\hspace{0.4cm}
\begin{minipage}[b]{0.4\linewidth}
\centering
\begin{tikzcd}
aB \arrow[d, "ec"'] \arrow[r, "Fc"] \arrow[rd, "jiFc"] & hG \arrow[d, "ji"] \\
eD \arrow[r, "jiFD"']                                  & jK                
\end{tikzcd}
\caption{A part of $\overline{\Q}^{\mathrm{Ba}}$ for $\Lambda$}
\label{exweakbrquivLambda}
\end{minipage}
\end{figure}
\end{example}

\begin{example}\label{nonassociativityweakbridge}
Consider the algebra $\Lambda'$ from Figure \ref{Nonassociative}. Here the bands are $\bb_1:=cbA$, $\bb_2:=biheD$, $\bb_3:=egF$ and $\bb_4:=lK$ along with their inverses. Consider the bridges $\uu_1=bA$, $\uu_2=eD$ and $\uu_3=Jih$. It is readily verified that the compositions $\uu_2\circ\uu_1=eDbA$, $\uu_3\circ\uu_2=JiheD$ and $(\uu_3\circ\uu_2)\circ\uu_1=JA$ exist but $\uu_3\circ(\uu_2\circ\uu_1)$ does not exist as $\uu_3(\uu_2\circ\uu_1)=JiheDbA$ contains a cyclic permutation of a band other than $\sbq(\uu_3)$.

\begin{figure}[h]
\begin{minipage}[b]{0.4\linewidth}
\centering
\begin{tikzcd}
                                   & v_1                  &                                                                 \\
v_3 \arrow[ru, "c"]                &                      & v_2 \arrow[ll, "b"] \arrow[lu, "a"']                            \\
v_4 \arrow[u, "d"] \arrow[rd, "e"] & v_7 \arrow[ru, "i"'] & v_8 \arrow[u, "j"']                                             \\
v_6 \arrow[u, "g"] \arrow[r, "f"]  & v_5 \arrow[u, "h"']  & v_9 \arrow[u, "k", bend left=49] \arrow[u, "l"', bend right=49]
\end{tikzcd}
\caption{$\Lambda'$ with $\rho=\{ai,bj,cd,jl,dg,hf,cbi,biheg\}$}
\label{Nonassociative}
\end{minipage}
\hspace{0.5cm}
\begin{minipage}[b]{0.5\linewidth}
\centering
\begin{tikzcd}
    &                                                                 & v_6                                 &                                      &                                                                 \\
    & v_5 \arrow[ru, "e"]                                             &                                     & v_4 \arrow[lu, "f"'] \arrow[ll, "d"] &                                                                 \\
    & v_7 \arrow[u, "g"] \arrow[rd, "h"']                             &                                     & v_3 \arrow[u, "c"']                  &                                                                 \\
v_0 & v_1 \arrow[l, "m", bend left=49] \arrow[l, "l"', bend right=49] & v_2 \arrow[ru, "b"'] \arrow[l, "a"] & v_8 \arrow[u, "i"]                   & v_9 \arrow[l, "k", bend left=49] \arrow[l, "j"', bend right=49]
\end{tikzcd}
\caption{$\Lambda''$ with $\rho=\{ma,ah,ci,ik,fc,eg,edcbh\}$}
\label{C-opposite}
\end{minipage}
\end{figure}
On the other hand the algebra $\Lambda''$ from Figure \ref{C-opposite} has bands $\bb'_1:=mL$, $\bb'_2:=edF$, $\bb'_3:=dcbhG$ and $\bb'_4:=kJ$ along with their inverses. Consider the bridges $\uu'_1=edcbAL$, $\uu'_2=dF$ and $\uu'_3=IbhG$. It is readily verified that the composition $\uu'_3\circ(\uu'_2\circ\uu'_1)$ exists but $(\uu'_3\circ\uu'_2)\circ\uu'_1$ does not exist.

These examples show that parentheses play a key role in the description of iterated compositions.
\end{example}

Say distinct strings $\xx_1,\xx_2$ with $s(\xx_1)=s(\xx_2)$ and $\sigma(\xx_1)=\sigma(\xx_2)$ \emph{fork} if the maximal common left substring of $\xx_1$ and $\xx_2$ is a proper left substring of both $\xx_1$ and $\xx_2$. We also say that a string $\xx$ \emph{forks} if there are distinct syllables $\alpha,\beta$ such that $\alpha\xx$ and $\beta\xx$ are strings.

Given a string $\xx_0$ of arbitrary finite length, a string $1_{(v,j)}$ of $0$ length, and $\bb\in Q^{\mathrm{Ba}}_0$, the notions of a (weak) half bridge $\xx_0\xrightarrow{\uu}\bb$, a (weak) reverse half bridge $\bb\xrightarrow{\uu}\xx_0$, and a (weak) zero bridge $\xx_0\to1_{(v,j)}$ are defined in a similar way as (weak) bridge; we refer the reader to \cite[\S~3.2]{GKS} where the details are written only for $\xx_0=1_{(v',j')}$. In contrast to Remark \ref{bridgehasposlen}, a (weak) half bridge/reverse half bridge/zero bridge could be of length $0$. For a weak half bridge $\xx_0\xrightarrow{\uu}\bb$ we set $\sbq(\uu):=\xx_0$ and $\tbq(\uu):=\bb$, for a weak zero bridge $\xx_0\xrightarrow{\uu'}1_{(v,j)}$ we set $\sbq(\uu'):=\xx_0$ and $\tbq(\uu'):=1_{(v,j)}(=1_{(t(\uu'),\epsilon(\uu'))})$, and similarly for a weak reverse half bridge.

Say a weak reverse half bridge $\bb\xrightarrow{\uu} 1_{(v,j)}$ is \emph{torsion} if $\uu\bb$ and $\bb^2$ fork and, for any weak bridge $\bb\xrightarrow{\uu'}\bb'$, the strings $\uu\bb$ and $\bb'\uu'\bb$ fork. Say a weak zero bridge $\xx_0\xrightarrow{\uu} 1_{(v,j)}$ is \emph{torsion} if $|\uu|>0$ and for any weak half bridge $\xx_0\xrightarrow{\uu'}\bb'$ the strings $\uu\xx_0$ and $\bb'\uu'\xx_0$ fork. Say that a weak reverse half bridge (resp. weak zero bridge) is \emph{maximal} if it is not a proper left substring of any other weak reverse half bridge (resp. weak zero bridge).

Say a syllable $\vv$ is an \emph{exit syllable} of band $\bb$ if $\vv\notin\bb$ and there is a cyclic permutation $\bb'$ of $\bb$ such that $\vv\bb'$ is defined. Denote the set of all exit syllables of $\bb$ by $\mathcal E(\bb)$.

\begin{rmk}\label{exitorder}
Suppose $\bb$ is a band and $\vv_1,\vv_2\in\mathcal E(\bb)$. If for some cyclic permutation $\bb'$ of $\bb$ the concatenations $\vv_1\bb'$ and $\vv_2\bb'$ exist then $\vv_1=\vv_2$.
\end{rmk}

The \emph{exit} of a (weak) bridge $\bb_1\xrightarrow{\uu}\bb_2$, denoted $\beta(\uu)$, between bands in a domestic string algebra is the syllable $\beta_k$ where $k$ is minimal in $\bb_2\uu\bb_1=\beta_n\hdots\beta_1\bb_1$  such that $\beta_k\hdots\beta_1\bb_1$ is not a substring of a power of $\bb_1$. Dually the \emph{entry} of $\uu$, denoted $\alpha(\uu)$, is $(\beta(\uu^{-1}))^{-1}$, where $\uu^{-1}$ is the bridge $\bb_2^{-1}\to\bb_1^{-1}$. The notions of the exit for a (weak) half bridge and that of the entry for a (weak) reverse half bridge are defined analogously, if such syllables exist. We have ensured that weak torsion reverse half bridges and weak torsion zero bridges have exits.

We give notations to some specific cyclic permutations of bands. For a weak bridge/weak half bridge $\uu$ such that $\alpha(\uu)$ exists denote by $\bua{\uu}$ the cyclic permutation of $\tbq(\uu)$ such that $\bua{\uu}\alpha(\uu)$ is a substring of $\tbq(\uu)^2\uu$. A weak half bridge $\uu$ for which $\alpha(\uu)$ does not exist is a left substring of a cyclic permutation of $\tbq(\uu)$; denote this permutation by $\bua{\uu}$. Dually for a weak bridge/weak reverse half bridge $\uu$ such that $\beta(\uu)$ exists denote by $\bub{\uu}$ the cyclic permutation of $\sbq(\uu)$ such that $\beta(\uu)\bub{\uu}$ is a string. A weak reverse half bridge $\uu$ for which $\beta(\uu)$ does not exist is a right substring of a cyclic permutation of $\sbq(\uu)$; denote this permutation by $\bub{\uu}$.

\begin{rmk}\label{exitisexitforarrow}
For each $\vv\in\mathcal E(\bb)$ either there is a weak bridge or a torsion weak reverse half bridge $\uu$ such that $\beta(\uu)=\vv$.
\end{rmk}

We classify weak bridges into two classes.
\begin{definition}\label{nabndef}
Say a weak bridge $\uu$ is \emph{normal} if there is a positive length substring of $\tbq(\uu)\uu\sbq(\uu)$ whose first syllable is $\beta(\uu)$ and last syllable is $\alpha(\uu)$. Otherwise say that $\uu$ is \emph{abnormal}.

If $\uu$ is normal, call the above mentioned substring its \emph{interior}, denoted $\uu^o$. On the other hand, if $\uu$ is abnormal then define its \emph{complimented interior}, denoted $\uu^c$, to be the substring (possibly of length $0$) of $\tbq(\uu)\uu\sbq(\uu)$ such that $\beta(\uu)\uu^c\alpha(\uu)$ is a string.

The interior of a weak reverse half bridge $\uu$ is defined to be the string satisfying $\uu\sbq(\uu)^\infty=\uu^o\bub{\uu}^\infty$.

For a band $\bb$ say that $\vv\in\mathcal E(\bb)$ is \emph{abnormal} if there is an abnormal weak bridge $\uu$ with $\sbq(\uu)=\bb$ and $\beta(\uu)=\vv$. Otherwise say that $\vv$ is normal.
\end{definition}
Note that for any normal or abnormal $\uu$, the set $\{s(\beta(\uu)),t(\alpha(\uu))\}$ is the set of the source and the target of its interior/complimented interior.

Abnormal arrows are a major source of the complexity of combinatorics of the strings. \S\ref{abnarrows} is devoted entirely to the study of these arrows.

Given a weak bridge $\uu$, set $\bar\lambda(\uu):=\mathcal E(\tbq(\uu))$. We partition this set by setting $\bar\lambda_i(\uu):=\{\vv\in\mathcal E(\tbq(\uu))\mid\theta(\vv)=i\}$ for $i\in\{1,-1\}$.

Given an exit syllable $\vv$ of a band $\bb$ let $\bar\lambda^b(\vv)$ and $\bar\lambda^r(\vv)$ denote the set of all weak bridges with exit $\vv$ and the set of all maximal torsion weak reverse half bridges with exit syllable $\vv$ respectively. Finally set $\bar\lambda(\vv):=\bar\lambda^b(\vv)\sqcup\bar\lambda^r(\vv)$.

For a string $\xx_0$ let $\bar\lambda^h(\xx_0)$ denote the set of all weak half bridges from $\xx_0$. Also let $\bar\lambda^z(\xx_0)$ denote the set of all maximal torsion weak zero bridges from $\xx_0$. We have a natural partition $\bar\lambda^h(\xx_0)=\bar\lambda^h_1(\xx_0)\sqcup\bar\lambda^h_{-1}(\xx_0)$ where $\bar\lambda^h_i(\xx_0):=\{\uu\in\bar\lambda^h(\xx_0)\mid\theta(\uu)=i\}$, where we assume that for a weak half bridge $\uu$ with $|\uu|=0$, we have $\theta(\uu):=\theta(\bua{\uu})=1$. Similarly we have a partition $\bar\lambda^z(\xx_0)=\bar\lambda^z_1(\xx_0)\sqcup\bar\lambda^z_{-1}(\xx_0)$ where $\bar\lambda^z_i(\xx_0):=\{\uu\in\bar\lambda^z(\xx_0)\mid\theta(\uu)=i\}$. Now set $\bar\lambda_i(\xx_0):=\bar\lambda^h_i(\xx_0)\sqcup\bar\lambda^z_i(\xx_0)$ for $i\in\{1,-1\}$.

The \emph{extended weak bridge quiver} $\overline{\Q}^{\mathrm{Ba}}(\xx_0)$ has as its vertices a subset of $Q_0^{\mathrm{Ba}}\sqcup\{\xx_0\}\sqcup\{\tbq(\uu)\mid\uu\mbox{ is a maximal torsion weak zero bridge/reverse half bridge}\}$ consisting of those reachable by a path starting from $\xx_0$, and as arrows the set of all weak bridges between bands, maximal torsion weak zero bridges, weak half bridges from $\xx_0$ and maximal torsion weak reverse bridges. The definition of the composition operation $\circ$ can be naturally extended to all arrows of the extended bridge quiver. Let $\overline{\Q}^{\mathrm{Ba}}_i(\xx_0)$ denote the full subquiver of $\overline{\Q}^{\mathrm{Ba}}(\xx_0)$ consisting of only those vertices which are reachable from $\xx_0$ via a path starting with some element of $\bar\lambda^h_i(\xx_0)$. The extended bridge quiver $\Q^{\mathrm{Ba}}(\xx_0)$ is the subquiver of $\overline\Q^{\mathrm{Ba}}(\xx_0)$ consisting of only bridges, (reverse) half bridges and zero bridges. The extended bridge quiver $\Q^{\mathrm{Ba}}_i(\xx_0)$ is the subquiver of $\overline{\Q}^{\mathrm{Ba}}_i(\xx_0)$ consisting of only those arrows which cannot be written as $\circ$-composition of any two arrows of $\overline{\Q}^{\mathrm{Ba}}_i(\xx_0)$.



Suppose $\Pp=(\uu_1,\uu_2\hdots\uu_n)$ and $\Pp'=(\uu'_1,\uu'_2\hdots\uu'_m)$ are paths in $\overline\Q^{\mathrm{Ba}}$ (or in $\overline\Q^{\mathrm{Ba}}(\xx_0)$) such that $\sbq(\Pp')=\tbq(\Pp)$. Then $\Pp+\Pp'$ is the concatenated path $(\uu_1,\uu_2\hdots\uu_n,\uu'_1,\uu'_2\hdots\uu'_m)$.

\begin{example}\label{parabnedgexample}
The extended weak bridge quiver of the algebra $X_1$ from Figure \ref{X1} is shown in Figure \ref{paraabnbrgeX1}. This example described by G. Puninski in \cite{Pun} consists of only two bands, namely $\bb:= acAB$ and $\bb^{-1}= baCA$, and two parallel abnormal bridges $\bb^{-1}\to\bb$, namely $\uu_1:=acA$ and $\uu_2:=aCA$. Here $\alpha(\uu_1)=b,\alpha(\uu_2)=C,\beta(\uu_1)=c$ and $\beta(\uu_2)=B$. There is one half bridge $1_{(v_2,1)}\xrightarrow{\uu_0}\bb^{-1}$ given by $\uu_0=ba$ and one weak half bridge $1_{(v_2,1)}\xrightarrow{a}\bb$ that factors as $\uu_2\circ\uu_0$.
\begin{figure}[h]
\begin{minipage}[b]{0.45\linewidth}
\centering
\begin{tikzcd}
v_1 \arrow["b"', loop, distance=2em, in=215, out=145] & v_2 \arrow[l, "a"] \arrow["c", loop, distance=2em, in=325, out=35]
\end{tikzcd}
\caption{$X_1$ with $\rho= \{b^2, c^2, bac\}$}
\label{X1}
\end{minipage}
\hspace{0.5cm}
\begin{minipage}[b]{0.5\linewidth}
\centering
\begin{tikzpicture}
\coordinate[label=above:$1_{(v_2,1)}$] (1_{(v_2,1)}) at (1.75,1.78);
\coordinate[label=below:$acAB$] (1_{(v_2,1)}) at (0,.4);
\coordinate[label=below:$baCA$] (1_{(v_2,1)}) at (3.3,.4);
\draw (1.5,1.82)[thick,->] -- (0,.3);
\coordinate[label=left:$a$] (a) at (.9,1.3);
\draw (1.8,1.82)[thick,->] -- (3.3,.3);
\coordinate[label=left:$ba$] (ba) at (3.1,1.3);
\draw (2.7,.2)[thick,->] -- (.7,.2);
\coordinate[label=above:$acA$] (acA) at (1.7,.2);
\draw (2.7,-.1)[thick,->] -- (.7,-.1);
\coordinate[label=below:$aCA$] (aCA) at (1.7,-.1);
\end{tikzpicture}    
\caption{$\overline{\Q}^{\mathrm{Ba}}(1_{(v_2,1)})$ for $X_1$}
    \label{paraabnbrgeX1}
    \end{minipage}
\end{figure}
\end{example}

\section{Abnormal arrows}\label{abnarrows}
If $\uu$ is an abnormal weak bridge then $\uu^c$ is a left substring of $\bua{\uu}$ and of a cyclic permutation, say $\bla{\uu}$, of $\sbq(\uu)$. Dually $\uu^c$ is a right substring of $\bub{\uu}$ and of a cyclic permutation, say $\blb{\uu}$, of $\tbq(\uu)$. If a weak bridge $\uu$ is abnormal then define $\uu^e$ to be the shortest left substring of $\bla{\uu}$ containing $\uu^c$ as a proper left substring such that $\uu^e\bua{\uu}$ is not a string. Note that $\uu^c$ could be of length $0$ but $|\uu^e|>0$.

We wish to understand the properties of maximal common substrings, possibly of length $0$, of two distinct bands, say $\bb_1,\bb_2$. Such strings are necessarily of the form $\uu^c$ for some weak bridge $\bb_1\xrightarrow{\uu}\bb_2$, for which $\bua{\uu}\bla{\uu}$ and $\blb{\uu}\bub{\uu}$ are strings.

\begin{example}
For the abnormal bridge $\uu_1$ in the bridge quiver of algebra $X_1$ from Example \ref{parabnedgexample}, we have $\uu_1^c=A$, $\bua{\uu_1}=BacA$, $\bla{\uu_1}=baCA$, $\blb{\uu_1}=ABac$ and $\bub{\uu_1}=AbaC$.
\end{example}

\begin{proposition}\label{constantsign}
If $\uu$ is an abnormal weak bridge then  $\delta(\uu^e)=-\theta(\beta(\uu))$.
\end{proposition}

\begin{proof}
Since $\Lambda$ is domestic the word $\bla{\uu}\bua{\uu}$ is not a string which could only happen if for some strings $\xx',\xx''$ with $|\xx'|,|\xx''|>0$ we have $\xx'\uu^c\xx''\in\rho\cup\rho^{-1}$, and $\xx'\uu^c,\uu^c\xx''$ are substrings of $\bla{\uu},\blb{\uu}$ respectively. Therefore $\uu^e$ is either a direct or an inverse string depending on whether $\beta(\uu)$ is inverse or direct respectively.
\end{proof}

Motivated by the above proof, for an abnormal $\uu$, let $\uu^\beta$ be the shortest left substring of $\bub{\uu}$ such that $\uu^e=\uu^\beta\uu^c$ and $\uu_\beta$ be the shortest right substring of $\bua{\uu}$ such that $\uu^\beta\uu^c\uu_\beta\in\rho\cup\rho^{-1}$. Similarly let $\uu^\alpha$ be the shortest right substring of $\bua{\uu}$ such that $\bla{\uu}\uu^\alpha$ is not a string and $\uu_\alpha$ be the shortest left substring of $\bub{\uu}$ such that $\uu_\alpha\uu^c\uu^\alpha\in\rho\cup\rho^{-1}$.
\begin{proposition}\label{component}
Let $\uu$ be an abnormal weak bridge. Suppose $\bub{\uu}=\xx_{2k}\hdots\xx_1$ where $\delta(\xx_j)=(-1)^{j+1}\theta(\beta(\uu))$. Then $\uu^c$ is a proper left substring (possibly of length $0$) of a unique $\xx_i$.
\end{proposition}

\begin{proof}
Since $\uu$ is abnormal, we have $\delta(\uu^e)=-\theta(\beta(\uu))$ from Proposition \ref{constantsign}. Therefore $\uu^c$ is a substring of some $\xx_i$.

If $\uu^c$ is a right substring of $\xx_i$ then $\xx_{i+1}\blb{\uu}$ is a string, and so is $\xx_i\hdots\xx_1\xx_{2k}\hdots\xx_{i+1}\blb{\uu}$ thus contradicting domesticity of $\Lambda$.

Let $\gamma$ be the first syllable of $\bla{\uu}$. From the above paragraph we know that $\theta(\gamma)=\theta(\xx_i)$. Let $\alpha,\beta$ denote the first and the last syllables of $\bua{\uu}$ respectively. Since $\bla{\uu}\bua{\uu}$ is not a string we have $\theta(\gamma)=\theta(\beta)$. Since $\alpha(\uu)$ is the entry syllable we have $\beta\neq\alpha(\uu)$. Since $\alpha\alpha(\uu)$ and $\alpha\beta$ are strings $\theta(\beta)\neq\theta(\alpha(\uu))$

If $\uu^c$ is not a left substring of $\xx_i$, then $\theta(\alpha(\uu))=\theta(\xx_i)$, and hence $\theta(\alpha(\uu))=\theta(\gamma)=\theta(\beta)$ which is a contradiction. Therefore our assumption is wrong and $\uu^c$ is a left substring of $\xx_i$.
\end{proof}

It follows from the above result and its dual that each string from $\{\uu^\alpha,\uu_\alpha,\uu^\beta,\uu_\beta\}$ has positive length, $|\uu^\alpha|\leq|\uu_\beta|$, and $|\uu^\beta|\leq|\uu_\alpha|$.

\begin{proposition}\label{banduniqueness}
Suppose $\uu_1,\uu_2$ are abnormal weak bridges with $\sbq(\uu_1)=\sbq(\uu_2)$. If $\uu_1^c$ is a left substring of $\uu_2^c$ and $\uu_2^c$ is a proper left substring of $\uu_1^e$ then $\tbq(\uu_1)=\tbq(\uu_2)$.
\end{proposition}
\begin{proof}
Let $\xx_1$ and $\xx_2$ be the complements of $\uu_1^c$ and $\uu_2^c$ in $\bua{\uu_1}$ and $\bua{\uu_2}$ respectively. Since $\uu_2^c$ is a proper left substring of $\uu_1^e$ the concatenation $\uu_2^c\bua{\uu_1}$ exists. Since $\theta(\xx_2)=-\theta(\uu_2^c)$ the further concatenation $\bua{\uu_2}\bua{\uu_1}=\xx_2\uu_2^c\bua{\uu_1}$ also exists. Dually since $\uu_1^c$ is a proper left substring of $\uu_2^c$, and hence of $\uu_2^e$, we can also show that the concatenation $\bua{\uu_1}\bua{\uu_2}$ exists. Since $\Lambda$ is domestic we conclude that $\tbq(\uu_1)=\tbq(\uu_2)$.
\end{proof}

\begin{rmk}
Suppose $\uu=\uu_2\circ\uu_1$ is an abnormal weak bridge. Then the above proposition and its dual together ensure that $|\uu_1^c|\geq|\uu^\beta\uu^c|$ whereas $|\uu_2^c|<|\uu^c\uu^\alpha|$.
\end{rmk}

\begin{corollary}\label{uniqueabnweakbridge}
If $\uu$ is an abnormal weak bridge then any weak bridge $\uu'\neq\uu$ satisfying $\sbq(\uu')=\sbq(\uu)$ and $\beta(\uu')=\beta(\uu)$ is a normal weak bridge.
\end{corollary}

\begin{proof}
Suppose a weak bridge $\uu'$ satisfying the hypotheses is abnormal. Then $\uu^c=\uu'^c$ is a proper left substring of both $\uu^e$ and $\uu'^e$. Hence from the above proposition $\uu'=\uu$. This shows that if $\uu'\neq\uu$ then $\uu'$ is normal.
\end{proof}

If $\bb$ is a band and $\vv\in\mathcal E(\bb)$ is abnormal then, in view of the above result, we use the notation $\bar\lambda^a(\vv)$ to denote the unique abnormal element of $\bar\lambda(\vv)$.

Given a band $\bb$ and $\vv,\vv'\in\mathcal E(\bb)$, say that $\vv$ is \emph{incident on} $\vv'$, written $\vv\perp\vv'$, if $\vv'$ is abnormal and there is a partition $\bar\lambda^a(\vv')^e=\xx_2\xx_1$ with $|\xx_2|>0$ such that $\vv\xx_1$ is a string. Note that $\vv=\vv'$ is possible. Also note that $\perp$ is a transitive relation.

\begin{rmk}\label{perpmaintainsamsign}
If $\vv\perp\vv'$ then $\theta(\vv)=-\theta(\bar\lambda^a(\vv')^e)=\theta(\vv')$.
\end{rmk}

\begin{proposition}\label{factregperp}
Suppose for a band $\bb$, $\vv,\vv'\in\mathcal E(\bb)$ satisfy $\vv\perp\vv'$. Then any $\uu\in\bar\lambda(\vv)$ factors (possibly trivially) through $\bar\lambda^a(\vv')$.
\end{proposition}

\begin{proof}
Suppose $\vv,\vv'$ are as in the hypotheses. Let $\uu'$ denote $\bar\lambda^a(\vv')$ for short. Let $\xx_1$ denote the proper left substring of ${\uu'}^e$ such that $\vv\xx_1$ is a string. Since $\xx_1\bua{\uu'}$ is a string, in view of the Remark \ref{perpmaintainsamsign} we see that $\vv\xx_1\bua{\uu'}\bla{\uu'}$ is also a string. Hence any $\uu\in\bar\lambda(\vv)$ factors through $\uu'$. This factorisation is trivial, i.e., $\uu=1_{(t(\uu),\varepsilon(\uu))}\circ\uu$ if and only if $\vv=\vv'$ and $\uu=\bar\lambda^a(\vv')$.
\end{proof}

\begin{example}
Consider the algebra $\Lambda'$ from Figure \ref{Nonassociative} in Example \ref{nonassociativityweakbridge}. Then $\vv:=J,\vv':=D\in\mathcal E(\bb_1)$ satisfy $\vv\perp\vv'$ and the weak bridge $(\uu_3\circ\uu_2)\circ\uu_1\in\bar\lambda(\vv)$ factors through $\bar\lambda^a(\vv')=\uu_1$.
\end{example}

We also note a couple of useful consequences of the above proposition.
\begin{rmk}\label{normalbridgenormalexit}
If $\uu$ is a normal bridge then $\beta(\uu)$ is a normal exit.
\end{rmk}

\begin{rmk}\label{factgivesperp}
Suppose $\uu$ is a weak bridge factors as $\uu=\uu''\circ\uu'$ where $\uu'$ is an abnormal weak bridge. Then $\beta(\uu)\perp\beta(\uu')$.
\end{rmk}

\begin{proposition}\label{comparisonofcomplement}
If an abnormal weak bridge $\bb_1\xrightarrow{\uu}\bb_2$ factors as $\uu=\uu_2\circ\uu_1$ through $\bb$ then $\uu_1$ is abnormal and $\uu^c$ is a substring of $\uu_1^c$.
\end{proposition}

\begin{proof}
Let $\uu=\uu'_2\uu'_1$ where the composition $\uu'_2\bb'\uu'_1$ is defined for some cyclic permutation $\bb'$ of $\bb$. If $\uu_1$ is normal then $\beta(\uu_1)\notin\bb'$ and hence $\beta(\uu_1)\in\uu'_1$. Let $\uu'_1=\uu_3\uu_4$ where $\beta(\uu_1)$ is the first syllable of $\uu_3$. Then $\uu_4$ is a right substring $\bub{\uu_1}$. As $\uu$ is abnormal the composition $\blb{\uu}\bub{\uu}$ is defined, and hence the first syllable, from left, of this composition that is not in $\bb_1$ is $\beta(\uu)$. Whereas the above discussion with $\uu=\uu'_2\uu'_1$ gives that such a syllable is $\beta(\uu_1)$. Hence $\beta(\uu)=\beta(\uu_1)$. The unique abnormal weak bridge with exit syllable $\beta(\uu)$, in view of Corollary \ref{uniqueabnweakbridge} is $\uu$ and $\bb\xrightarrow{\uu_2}\bb_2$ is a weak bridge. On the other hand, since $\beta(\uu)\perp\beta(\uu)$, Proposition \ref{factregperp} guarantees the existence of a weak bridge $\bb_2\to\bb$ contradicting domesticity of $\Lambda$. Hence $\uu_1$ is abnormal and $\beta(\uu_1)\neq\beta(\uu)$. We also see that $\uu^c$ and $\uu_1^c$ are comparable strings.

Now either $\uu=\uu_2\circ\uu_1=\uu_2\uu_1$ or $\uu_2\uu_1=\uu'_2\bb'\uu'_1$. In the former case $\ww\uu_1^c$ and $\uu^c$ have a common right substring for some $\ww$ with $|\ww|>0$. In the latter case either a similar statement holds or there is a string $\ww'$ with $|\ww'|>0$ such that $\ww'\uu^c$ and $\uu_1^c$ have a common right substring.

If $\ww\uu_1^c$ and $\uu^c$ have a common right substring for some $\ww$ with $|\ww|>0$ then, in view of Proposition \ref{constantsign}, patching signature types we have $|\uu^c|>0$ and $\uu^c$ is a substring of $\uu_1^e$. Thus by Proposition \ref{banduniqueness} we have $\bb_2=\bb$. This is a contradiction to $|\ww|>0$. Therefore there is a string $\ww'$ with $|\ww'|>0$ such that $\ww'\uu^c$ and $\uu_1^c$ have a common right substring which in turn implies, with similar arguments, that $\uu^c$ is a left substring of $\uu_1^c$.
\end{proof}

Suppose $\vv\perp\vv'$ and $\vv\neq\vv'$ are both abnormal. Let $j:=\theta(\vv)$ then $\theta(\vv')=j$ in view of Remark \ref{perpmaintainsamsign}. Proposition \ref{factregperp} guarantees that $\bar\lambda^a(\vv)$ factors through $\bar\lambda^a(\vv')$. Then Proposition \ref{comparisonofcomplement} guarantees that $\bar\lambda^a(\vv)^c$ is a substring of $\bar\lambda^a(\vv')^c$. Furthermore Proposition \ref{banduniqueness} shows that $\bar\lambda^a(\vv)^e$ is a proper substring of $\bar\lambda^a(\vv')^e$.

Suppose $\bb$ is a band and $(\vv_1,\hdots,\vv_n)$ is a maximal chain of abnormal elements of $\mathcal E(\bb)$ such that $\vv_j\perp\vv_{j+1}$ for each $1\leq j<n$. For each $1\leq j\leq n$, let $\bb_j:=\tbq(\bar\lambda^a(\vv_j))$. Let $\bb\xrightarrow{\uu_j}\bb_j$ denote the canonical weak bridge.

Proposition \ref{factregperp} guarantees that, for each $1\leq j<n$, $\bua{\bar\lambda^a(\vv_j)}\bua{\bar\lambda^a(\vv_{j+1})}$ is a string, and hence there is a canonical weak bridge $\bb_{j+1}\xrightarrow{\uu'_j}\bb_j$. It is straightforward to verify that, for each $1\leq j\leq n$, the weak bridge $\uu_j$ factors as the composition $\bb\xrightarrow{\uu_n}\bb_n\xrightarrow{\uu'_{n-1}}\bb_{n-1}\xrightarrow{\uu'_{n-2}}\hdots\xrightarrow{\uu'_j}\bb_j$ of weak bridges.

\begin{corollary}\label{interweakbridgisbridg}
Suppose $\bb$ is a band and $(\vv_1,\hdots,\vv_n)$ is a maximal chain of abnormal elements of $\mathcal E(\bb)$ such that $\vv_j\perp\vv_{j+1}$ for each $1\leq j<n$. Using the notation above, $\uu_n$ and $\uu'_j$ are bridges for each $1\leq j<n$.
\end{corollary}

\begin{proof}
The contrapositive of Proposition \ref{comparisonofcomplement} applied to $\uu=\uu_n$ together with maximality of $\uu_n$ with respect to $\perp$ shows that $\uu_n$ is a bridge.

Now consider $\uu_{j+1}$ and recall that $|(\uu_{j+1})_{\beta}|>0$. Then ${\uu'}_j^c=\uu_j^c(\uu_{j+1})_{\beta}$ and since $\Lambda$ is domestic the concatenation $\uu_{j+1}^c(\uu_{j+1})_{\beta}\bla{\uu'_j}$ does not exist. Therefore ${\uu'}_j^e$ is a left substring of $\uu_{j+1}^c(\uu_{j+1})_{\beta}$.

If $\uu'_j$ factors as $\uu''\circ\uu'$ then Proposition \ref{comparisonofcomplement} guarantees that $\uu'$ is abnormal and ${\uu'_j}^c$ is a proper left substring of ${\uu'}^c$. It is easy to see that $\theta(\beta(\uu'))=-\theta({\uu'}^c)=-\theta(\uu_{j+1}^c)=\theta(\beta(\uu_{j+1}))$. Hence $\beta(\uu')\in\bar\lambda(\uu)$. From the above paragraph ${\uu'_j}^e$ is a substring of $\uu_{j+1}^c(\uu_{j+1})_{\beta}$ and hence $\vv_j\perp\beta(\uu'),\beta(\uu')\perp\vv_{j+1}$ where $\beta(\uu')$ is distinct from $\vv_j$ and $\vv_{j+1}$, which contradicts the hypothesis that $\vv_{j+1}$ is the immediate successor of $\vv_j$ with respect to $\perp$. This completes the proof that $\uu'_j$ is a bridge.
\end{proof}

The following is clear from the above corollary and Proposition \ref{factregperp}.

\begin{rmk}\label{uefactorization}
Using the notations of the above corollary, suppose $\vv\in\mathcal E(\bb)$ is distinct from any $\vv_j$ and that $\vv^\perp:=\{\vv_j\mid\vv\perp\vv_j\}\neq\emptyset$. Then for any $\vv_j\in\vv^\perp$ any $\uu'\in\bar\lambda(\vv)$ factors as $\uu''_j\circ\uu_j$. Moreover, $\uu''_j$ does not factor through $\uu''_k$ for any $j<k,\vv_k\in\vv^\perp$.
\end{rmk}

Now we are ready to give a characterization of abnormal weak bridges.
\begin{proposition}\label{characterizeabnormality}
A weak bridge $\uu$ is abnormal if and only if for any factorization $\uu=\uu_2\circ\uu_1$ we have that $\uu_2$ is abnormal and $\beta(\uu)=\beta(\uu_2)$.
\end{proposition}

\begin{proof}
Suppose a weak bridge $\uu$ has the property that for any factorization $\uu=\uu_2\circ\uu_1$ we have that $\uu_2$ is abnormal and $\beta(\uu)=\beta(\uu_2)$. Consider any such factorization. Since $\uu_2$ is abnormal $\beta(\uu_2)=\beta(\uu)$ is a syllable of $\tbq(\uu_2)=\tbq(\uu_1)$. Hence $\uu$ is abnormal.

Conversely suppose $\uu$ is abnormal. Then for any factorization $\uu=\uu_2\circ\uu_1$ Proposition \ref{comparisonofcomplement} guarantees that $\uu_1$ is abnormal and that $\uu^c$ is a left substring of $\uu_1^c$. Since $\uu\neq\uu_1$ we see that $\uu^c$ is a proper left substring of $\uu_1^c$. Moreover $\sbq(\uu_2)$ and $\tbq(\uu_2)$ intersect as $\uu^c$ is a common substring of $\sbq(\uu_2)=\tbq(\uu_1)$ and $\tbq(\uu_2)= \tbq(\uu)$. Therefore $\beta(\uu_2)$ is a syllable of $\tbq(\uu_2)$ and hence $\uu_2$ is abnormal. Furthermore $\uu^c$ is a substring of $\uu_2^c$. Clearly $\beta(\uu)$ is also an exit syllable of $\sbq(\uu_2)$ otherwise, since $\beta(\uu)\uu^c$ is a string and $\uu^c$ is a proper left substring of $\uu_1^c$, we will get that $\beta(\uu)$ is a syllable of $\sbq(\uu_1)=\sbq(\uu)$, a contradiction. Since $\uu^c$ is a substring of $\uu_2^c$ and $\beta(\uu)\uu^c$ is a string, from the definition of the exit syllable of a weak bridge we also have that $\beta(\uu)\uu_2^c$ is a string and hence $\beta(\uu)=\beta(\uu_2)$.
\end{proof}

The following is a very useful observation.
\begin{proposition}\label{normalexitnormalexit}
If $\uu=\uu_2\circ\uu_1$ and $\uu_1$ is normal then $\beta(\uu)=\beta(\uu_1)$. Consequently the conclusion also holds when $\beta(\uu_1)$ is normal.
\end{proposition}

\begin{proof}
Since $\uu_1$ is normal,  $\beta(\uu_1)$ is the first syllable of $\uu_1^o$ but not a syllable of $\tbq(\uu_1)$.

If $\uu=\uu_2\uu_1$ then clearly $\beta(\uu_1)$ is the first syllable of $\uu$ not in $\sbq(\uu)$, and hence $\beta(\uu)=\beta(\uu_1)$.

If $\uu_2\uu_1=\ww_2\bb'\ww_1$ and $\uu=\ww_2\ww_1$, where $\bb'$ is a cyclic permutation of $\tbq(\uu_1)$ and the last syllable of $\ww_1$ is not a syllable of $\tbq(\uu_1)$ then $\beta(\uu_1)$ is a syllable of $\ww_1$, and hence $\beta(\uu)=\beta(\uu_1)$.
\end{proof}

Composition of bridges satisfies a right cancellation property.
\begin{proposition}\label{leftcancel}
Suppose $\uu_2\circ\uu_1=\uu'_2\circ\uu_1$. Then $\uu_2=\uu'_2$.
\end{proposition}

\begin{proof}
Recall that $\uu_2\circ\uu_1$ is either $\uu_2\uu_1$ or $\Red{\tbq(\uu_1)}(\uu_2\uu_1)$. Note that $\uu_2\circ\uu_1=\uu_2\uu_1$ if and only if $\uu_1$ is a left substring of $\uu_2\circ\uu_1$ if and only if $\uu_2$ is a right substring of $\uu_2\circ\uu_1$. Since the same statement also holds when $\uu_2$ is replaced by $\uu'_2$ we get that $\uu_2\circ\uu_1=\uu_2\uu_1$ if and only if $\uu'_2\circ\uu_1=\uu'_2\uu_1$.

If $\uu_2\uu_1=\uu_2\circ\uu_1=\uu'_2\circ\uu_1=\uu'_2\uu_1$ then clearly $\uu_2=\uu'_2$.

On the other hand, if $\uu_2\circ\uu_1=\Red{\tbq(\uu_1)}(\uu_2\uu_1)=\Red{\tbq(\uu_1)}(\uu'_2\uu_1)$ then there is a maximal left substring $\ww$ of $\uu_2\circ\uu_1=\uu'_2\circ\uu_1$ for which $\bb'\ww$ is a string for some cyclic permutation $\bb'$ of $\tbq(\uu_1)$. Then for some strings $\bar\ww,\bar\ww'$ we have $\uu_2\uu_1=\bar\ww\bb'\ww$ and $\uu'_2\uu_1=\bar\ww'\bb'\ww$. Since $\bar\ww\ww=\uu_2\circ\uu_1=\uu'_2\circ\uu_1=\bar\ww'\ww$ we get $\bar\ww=\bar\ww'$, and it follows that $\uu_2=\uu'_2$.
\end{proof}

The above cancellation property is paired with the following result to be useful.
\begin{proposition}\label{essentialexit}
Suppose $\uu$ is a weak bridge and $\uu=\uu_2\circ\uu_1$ for some abnormal bridge $\uu_1$ then for any factorization $\uu=\uu'_2\circ\uu'_1$ with $\uu'_1$ a bridge we have $\uu'_1=\uu_1$.
\end{proposition}

\begin{proof}
Since $\uu=\uu_2\circ\uu_1$ and $\beta(\uu_1)$ is abnormal we have $\beta(\uu_1)\in\beta(\uu)^\perp$ by Remark \ref{factgivesperp}.

If $\uu'_1$ is normal then $\beta(\uu'_1)$ is normal by Remark \ref{normalbridgenormalexit}. Then by Proposition \ref{normalexitnormalexit} we have $\beta(\uu)=\beta(\uu'_1)$. But $\beta(\uu'_1)^\perp=\emptyset$ by Proposition \ref{factregperp}, a contradiction. Therefore $\uu'_1$ is abnormal.

Since $\beta(\uu)$ is incident on both $\beta(\uu_1)$ and $\beta(\uu'_1)$ there is an incidence relation between the latter two. Since both $\uu_1,\uu'_1$ are bridges we get $\beta(\uu_1)=\beta(\uu'_1)$ and hence $\uu_1=\uu'_1$ by Corollary \ref{uniqueabnweakbridge}.
\end{proof}

Combining the above two results we obtain the following.
\begin{corollary}\label{abnwbcancel}
Suppose $\uu=\uu_n\circ(\hdots\circ(\uu_2\circ\uu_1)\hdots)$, where $\uu_j$ is an abnormal bridge for each $1\leq j<n$ and $\uu_n$ is a bridge. Then the factorization of $\uu$ into bridges is unique.
\end{corollary}

The next observation will be useful later.
\begin{rmk}\label{associativenaafact}
Suppose $\uu=\uu_n\circ(\hdots\circ(\uu_2\circ\uu_1)\hdots)$, where $\uu_j$ is an abnormal bridge for each $1\leq j<n$ and $\uu_n$ is a bridge. Then such factorization, which is unique thanks to Corollary \ref{abnwbcancel}, is associative.
\end{rmk}

In the rest of the section we study abnormal weak half bridges.
\begin{definition}
Say a weak half bridge $\xx_0\xrightarrow{\uu}\bb$ is \emph{abnormal} if $\uu$ is a right substring of $\bb$. Otherwise say that $\uu$ is \emph{normal}.
\end{definition}

The entry syllable, $\alpha(\uu)$, of a normal half bridge is the last syllable from the left in $^\infty\tbq(\uu)\uu\xx_0$ that is not a syllable of $\tbq(\uu)$. Note that $\alpha(\uu)\in\uu$. Let $\bua{\uu}$ denote the unique cyclic permutation of $\tbq(\uu)$ for which  $^\infty\tbq(\uu)\uu\xx_0=\ ^\infty\bua{\uu}\ww\xx_0$, for a string $\ww$ of positive length that has $\alpha(\uu)$ as the last syllable. The string $\ww$ is the \emph{interior} of $\uu$, and we denote it by $\uu^o$.

If $\uu$ is abnormal then $\bua{\uu}$ denotes the unique cyclic permutation of $\tbq(\uu)$ for which $^\infty\tbq(\uu)\uu\xx_0=\ ^\infty\bua{\uu}\xx_0$.

We begin our analysis of half bridges with a useful proof technique.
\begin{proposition}\label{squaresubstring}
Suppose $\bb$ is a band, $\uu$ is a cyclic string and $\delta(\uu)\neq0$. Then $\uu^2$ is not a substring of $\bb$.
\end{proposition}

\begin{proof}
Suppose, for contradiction, that $\delta(\uu)\neq0$ but $\uu^2$ is a substring of $\bb$.

Consider a cyclic permutation $\bb'$ of $\bb$ such that $\bb'=\uu^2\uu'$. Then $\delta(\uu\uu')=0$. Then $\bb'':=\uu\uu'$ is a cyclic string such that ${\bb''}^2$ exists. Hence $\bb''$ is a cyclic permutation of a power of a band by \cite[Remark~3.1.3]{GKS}. Moreover both $\bb'\bb''$ and $\bb''\bb'$ exist, a contradiction to the domesticity of $\Lambda$, thus completing the proof.
\end{proof}

Using arguments similar to the proof of the above proposition we can prove the following.
\begin{proposition}\label{uniquebanduniqueabnormalhalfbridge}
Suppose $i\in\{1,-1\}$, $\uu_1,\uu_2\in\bar\lambda^h_i(\xx_0)$ are abnormal and $\tbq(\uu_1)=\tbq(\uu_2)$. Then $\uu_1=\uu_2$.
\end{proposition}

\begin{proof}
Since $\uu_1,\uu_2$ are abnormal they are right substrings of $\bb:=\tbq(\uu_1)=\tbq(\uu_2)$. Without loss of generality we may assume that $\uu_1=\uu_2\uu$ for some string $\uu$.

Suppose, for contradiction, that $|\uu|>0$. Let $\uu'$ be the maximal common left substring of $\uu$ and $\uu_2$. Since $\theta(\bb\uu_2)=\theta(\bb\uu_1)$ we have that $|\uu'|>0$. Let $\uu=\uu''\uu'$ and $\uu_2=\uu'''\uu'$. Let $\uu_3$ denote the left substring of $\bb$ such that $\uu_1\uu_3=\bb$. Since $\uu'$ forks, we get $\theta(\uu'')=-\theta(\uu''')$.

If $\delta(\uu')=0$ then the concatenation $\uu_2\uu_3\uu_2\uu_3$ exists. By \cite[Remark~3.1.3]{GKS} we obtain that $\bb':=\uu_2\uu_3$ is a power of a cyclic permutation of a band for which both $\bb\bb'$ and $\bb'\bb$ are strings, thus contradicting the domesticity of $\Lambda$. Therefore there are strings $\uu_4,\uu_5$ of positive length such that $\uu_5\uu'\uu_4\in\rho\cup\rho^{-1}$. In particular, $\uu_5$ is a left substring of $\uu'''$ and $\theta(\uu''')=\theta(\uu_5)=\delta(\uu')$.

Further if $\uu^2$ is a string then from \cite[Remark~3.1.3]{GKS} we see that $\bb''=\uu_2\uu^2\uu_3$ is a cyclic permutation of a band for which $\bb''\bb$ and $\bb\bb''$ are strings contradicting the domesticity of $\Lambda$. Thus $\uu^2$ is not a string. Therefore there are strings $\uu_6,\uu_7$ of positive length such that $\uu_7\uu'\uu_6\in\rho\cup\rho^{-1}$. In particular, $\uu_7$ is a left substring of $\uu''$ and $\theta(\uu'')=\theta(\uu_7)=\delta(\uu')$.

The conclusions of the above three paragraphs are contradictory. Hence our assumption $|\uu|>0$ is wrong, and we have concluded $\uu_1=\uu_2$.
\end{proof}

The following example shows that the above result fails when $\theta(\uu_1)\neq\theta(\uu_2)$.
\begin{example}
Consider the algebra $\Lambda^{(iii)}$ from Figure \ref{Disthalfabnbrge}. 
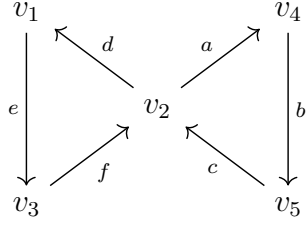
\begin{figure}[h]
    \centering
\begin{tikzcd}
	{v_1} && {v_4} \\
	& {v_2} \\
	{v_3} && {v_5}
	\arrow["e"', from=1-1, to=3-1]
	\arrow["d"', from=2-2, to=1-1]
	\arrow["f"', from=3-1, to=2-2]
	\arrow["a", from=2-2, to=1-3]
	\arrow["b", from=1-3, to=3-3]
	\arrow["c", from=3-3, to=2-2]
\end{tikzcd}    
\caption{$\Lambda^{(iii)}$ with $\rho=\{af,dc,edf,(bac)^2\}$}
    \label{Disthalfabnbrge}
\end{figure}
Here the bands are $\bb:=cbaDEF$ and $\bb^{-1}$. If $\xx_0:=D$ then there are two half bridges $\xx_0\to\bb$, namely $\uu_1=cba$ and $\uu_2=1_{(v_2,\epsilon(D))}$. It is readily verified that both $\uu_1$ and $\uu_2$ are abnormal and that $\theta(\uu_1)=-\theta(\uu_2)$.
\end{example}

Below we note an interesting observation about injectivity of the (complemented) interior map.
\begin{proposition}\label{interiorisinjective}
Suppose two normal weak bridges (resp. abnormal weak bridges, weak reverse half bridges, normal weak half bridges) $\uu_1,\uu_2$ satisfy $\sbq(\uu_1)=\sbq(\uu_2)$, $\tbq(\uu_1)=\tbq(\uu_2)$ and $\uu_1^o=\uu_2^o$ (resp. $\uu_1^c=\uu_2^c$, $\uu_1^o=\uu_2^o$, $\uu_1^o=\uu_2^o$) then $\uu_1=\uu_2$.
\end{proposition}

\begin{proof}
If $\uu_1,\uu_2$ are normal weak bridges satisfying the hypotheses, then
$$^\infty\tbq(\uu_1)\uu_1\sbq(\uu_1)^\infty=\ ^\infty\bua{\uu_1}\uu_1^o\bub{\uu_1}^\infty=\ ^\infty\bua{\uu_2}\uu_2^o\bub{\uu_2}^\infty=\ ^\infty\tbq(\uu_2)\uu_2\sbq(\uu_2)^\infty.$$
Hence the conclusion follows.

When $\uu_1,\uu_2$ are abnormal then the conclusion follows from Proposition \ref{banduniqueness}.

The remaining two cases have proofs similar to the case of normal weak bridges.
\end{proof}

\begin{rmk}\label{casebycaseanalysis}
We close this section with a description of (complemented) interiors of $\uu=\uu_2\circ\uu_1$ in terms of $\uu_1^{o/c}$ and $\uu_2^{o/c}$. We do not include proofs of these statements but only indicate certain special cases.
	
\noindent{}\textbf{Case I}: If both $\uu_1$ and $\uu_2$ are normal then $\uu$ is normal using Proposition \ref{characterizeabnormality}. Let $\ww$ be the shortest left substring of $\bua{\uu_1}$ such that the concatenation $\uu_2^o\ww\uu_1^o$ exists. Then $\uu^o=\uu_2^o\ww\uu_1^o$.

\noindent{}\textbf{Case II}: If $\uu_1$ is normal but $\uu_2$ is abnormal then $\uu$ is normal.
	
Since $\uu_2$ is abnormal, $\alpha(\uu_2)$ is a syllable of $\tbq(\uu_1)$ and hence $\alpha(\uu_1)\neq \alpha(\uu_2)$. Let $\ww$ denote the smallest left substring of $\bua{\uu_1}$ such that $\beta(\uu_2)\ww\alpha(\uu_1)$ is a string. Clearly $0\leq |\ww|<|\bua{\uu_1}|$. Moreover $\ww$ and $\uu_2^c$ are comparable as both of them are right substrings of $\bub{\uu_1}$. Let $\beta,\gamma$ denote the first and the last syllable of $\bua{\uu_2}$ respectively. Now consider the following cases:
\begin{enumerate}
    \item $\uu_2^c=\ww\ww'$ with $|\ww'|>0$. In this case $\theta(\uu_2^c)=\theta(\beta)=\theta(\gamma)=-\theta(\beta(\uu_2))$. Thus the only possible choice of $t(\uu_2)$ in $\uu_2^e$ is $s(\beta(\uu_2))$. Therefore $\uu_2\uu_1$ contains a cyclic permutation of $\tbq(\uu_1)$ as a substring. Therefore $\alpha(\uu)= \alpha(\uu_1)$ and hence $\uu^o=\uu_1^o$. 
    
    \item $\ww=\uu_2^c$. In this case $\beta$ is also the first syllable of $\bua{\uu_1}$ and hence $\beta\alpha(\uu_1)$ is a string. Since $\beta\gamma$ and $\beta\alpha(\uu_2)$ are defined we have $\gamma=\alpha(\uu_1)$ by the definition of a string algebra. Let $\ww''$ denote the maximal common right substring of $\bua{\uu_2}$ and $\uu_1^o$. Then $\alpha(\uu_1)$ is a syllable of $\ww''$ and $\uu_1^o=\ww''\uu^o$.
    
    \item $\ww= \uu_2^c\ww'$ with $|\ww'|>0$. In this case, $\uu^o=\ww'\uu_1^o$.
\end{enumerate}

\noindent{}\textbf{Case III}: If $\uu_2$ is normal but $\uu_1$ is abnormal then $\uu$ is normal.
	
Suppose $\uu_1^e=\xx_2\xx_1$ with $|\xx_2|>0$ such that $\beta(\uu)\xx_1$ is a string. There are three different possibilities for $\uu^o$ depending on the length of $\xx_1$.
\begin{enumerate}
    \item If $|\xx_1|<|\uu_1^c|$ then $\uu^o=\uu_2^o$.
    
    \item If $|\xx_1|=|\uu_1^c|$ then $\uu^o=\uu_2^o\ww$ where $\ww$ is the shortest left substring of positive length of $\blb{\uu_1}$ such that the above concatenation is possible.
    
    \item If $|\xx_1|>|\uu_1^c|$ then $\uu^o=\ww$ where $\ww$ is the longest right substring of $\uu_2^o$ such that $\ww\xx_1$ is a string. In this case $|\ww|>0$ is guaranteed as $\Lambda$ is domestic.
\end{enumerate}

\noindent{}\textbf{Case IV}: If both $\uu_1,\uu_2$ are abnormal then $\uu$ could be normal or abnormal as described in Proposition \ref{characterizeabnormality}.
\begin{enumerate}
    \item If $\uu$ is abnormal, then $\uu^c=\ww$ satisfies $\uu_1^c=\ww_1\ww$ and $\uu_2^c=\ww\ww_2$  for some strings $\ww_1,\ww_2$. Note that $\ww$ could have length $0$.

    \item If $\uu$ is normal, then $\uu^o=\ww$ where $\ww$ is the shortest left substring of $\blb{\uu_1}$ such that $\uu_2^c\ww\uu_1^c$ is a string. In this case $|\ww|>0$ is guaranteed as $\Lambda$ is domestic.
\end{enumerate}
\end{rmk}

\begin{example}
\begin{figure}[h]
\begin{minipage}[b]{0.43\linewidth}
\centering
\begin{tikzcd}
	&& {v_3} \\
	{v_1} & {v_2} && {v_4} \\
	&& {v_4}
	\arrow["b", from=2-2, to=1-3]
	\arrow["e", from=3-3, to=2-2]
	\arrow["a"', from=2-2, to=2-1]
	\arrow["d"', from=2-4, to=2-2]
	\arrow["c"', from=2-4, to=1-3]
	\arrow["f"', from=3-3, to=2-4]
\end{tikzcd}
    \caption{$\Lambda^{(iv)}$ with $\rho=\{ad,be,cf,bdf\}$}
    \label{C-equal}
    \end{minipage}
\hspace{0.7cm}
\begin{minipage}[b]{0.5\linewidth}
\centering
\begin{tikzcd}
                                                                   &                                     & v_2 \arrow[r, "a"]                    & v_1                                    \\
v_6 \arrow[rru, "c"]                                               & v_5 \arrow[l, "f"]                  & v_4 \arrow[l, "e"']                   & v_3 \arrow[l, "d"'] \arrow[lu, "b"']   \\
v_{12} \arrow[u, "n"]                                              & v_7 \arrow[u, "g"] \arrow[rd, "h"'] & v_{11} \arrow[u, "k"] \arrow[r, "l"'] & v_{10} \arrow[u, "j"'] \arrow[ld, "i"] \\
v_{13} \arrow[u, "o"', bend right=49] \arrow[u, "p", bend left=49] & v_8 \arrow[u, "m"]                  & v_9                                   &                                       
\end{tikzcd}
    \caption{$\Lambda^{(v)}$ with $\rho=\{ac,bj,ek,fg,cn,il,np,hm,edjl,cfedj\}$}
    \label{case3analysis}
    \end{minipage}
\end{figure}

The examples of various cases described above are described below.
\begin{itemize}
    \item[II(1)] $1_{(v_1,1)}\xrightarrow{bA}bdC\xrightarrow{dC}dfE$ in the algebra $\Lambda^{(iv)}$ from Figure \ref{C-equal}.
    \item[II(2)] $mL\xrightarrow{edcbAL} edF\xrightarrow{dF} dcbhG$ in $\Lambda''$ from Figure \ref{C-opposite}.
    \item[II(3)] $1_{(v_1,1)}\xrightarrow{A}cfedB\xrightarrow{edB}edjIhG$ in the algebra $\Lambda^{(v)}$ from Figure \ref{case3analysis}.
    \item[III(2)] $cfedB\xrightarrow{edB}edjIhG\xrightarrow{MG}1_{(v_8,1)}$ in the algebra $\Lambda^{(v)}$ from Figure \ref{case3analysis}.
    \item[III(3)] $cfedB\xrightarrow{edB}edjIhG\xrightarrow{Nf}pO$ in the algebra $\Lambda^{(v)}$ from Figure \ref{case3analysis}.
    \item[IV(1)] $cfedB\xrightarrow{edB}edjIhG\xrightarrow{djIhG}djlK$ in the algebra $\Lambda^{(v)}$ from Figure \ref{case3analysis}.
    \item[IV(2)] $cbA\xrightarrow{bA}biheD\xrightarrow{eD}egF$ in the algebra $\Lambda'$ from Figure \ref{Nonassociative}.
\end{itemize}
\end{example}

\section{H-equivalence and H-reduction}\label{secweakarchbrquiv}
Say that two strings $\yy_1,\yy_2$ are \emph{H-equivalent}, written $\yy_1\equiv_H\yy_2$, if
\begin{itemize}
    \item $t(\yy_1)=t(\yy_2)$;
    \item for each string $\xx$ such that $s(\xx)=t(\yy_1)$, $\xx\yy_1$ is a string if and only if $\xx\yy_2$ is so.
\end{itemize}

\begin{rmk}
If $\yy_1\equiv_H\yy_2$ then $\xx\yy_1\mapsto\xx\yy_2:H_l(\yy_1)\to H_l(\yy_2)$ is an isomorphism.
\end{rmk}

The following observation will be useful later.
\begin{proposition}\label{newExtHequivcriterion}
For strings $\yy_1,\yy_2$, if $\yy_1\equiv_H\yy_2$ then $\xx\yy_1\equiv_H\xx\yy_2$ for any string $\xx$ such that $\xx\yy_1$ and $\xx\yy_2$ are strings.
\end{proposition}

\begin{proof}
Suppose $\xx\yy_1$ and $\xx\yy_2$ are both strings. If $(\zz\xx)\yy_1$ is a string then so is $(\zz\xx)\yy_2$ using the H-equivalence $\yy_1\equiv_H\yy_2$. Using the above statement and the dual obtained by swapping $\yy_1$ and $\yy_2$, we readily get the conclusion.
\end{proof}

We introduce some notation so that we can provide a ``checkable'' criterion for H-equivalence between two strings. For a string $\yy$ if there is a right substring $\ww$ of positive length such that $\ww$ is a proper left substring of a word in $\rho\cup\rho^{-1}$ then we set $\rho_r(\yy)$ to be the maximal such right substring of $\yy$. If no such right substring exists then we set $\rho_r(\yy):=1_{(t(\yy),\epsilon(\yy))}$. Clearly either $|\rho_r(\yy)|=0$ or $\delta(\rho_r(\yy))\neq0$. Dually we define $\rho_l(\yy):=(\rho_r(\yy^{-1}))^{-1}$.

\begin{proposition}\label{newHequivcriterion}
For strings $\yy_1,\yy_2$ with $t(\yy_1)=t(\yy_2)$, let $\gamma_j$ denote the last syllable of $\yy_j$, if exists, for $j\in\{1,2\}$. Then $\yy_1\equiv_H\yy_2$ if and only if one of the following happens:
\begin{enumerate}
    \item $|\rho_r(\yy_1)|=|\rho_r(\yy_2)|=0$;
    \item $|\rho_r(\yy_1)||\rho_r(\yy_2)|>0$ and 
    \begin{enumerate}
        \item if $\gamma_1=\gamma_2$ then $\rho_r(\yy_1)=\rho_r(\yy_2)$;
        \item if $\gamma_1\neq\gamma_2$ then $\gamma_j^{-1}\rho_r(\yy_k)\in\rho\cup\rho^{-1}$ for $k\neq j$ and whenever the word $\xx\yy_j$ is not a string then $\gamma_k^{-1}$ is the first syllabe of $\xx$.
    \end{enumerate}
\end{enumerate}
\end{proposition}

\begin{proof}
We first show the forward direction. Suppose $\yy_1\equiv_H\yy_2$. 

\noindent{}\textbf{Claim:} $|\rho_r(\yy_j)|>0$ if and only if $|\rho_r(\yy_k)|>0$.

Without loss assume that $|\rho_r(\yy_1)|>0$ but $|\rho_r(\yy_2)|=0$. In this case, $\rho_r(\yy_1)$ is not a right substring of $\yy_2$. Let $\xx$ be the string such that $\xx\rho_r(\yy_1)\in\rho\cup\rho^{-1}$. Since elements of $\rho$ are incomparable we see that $\xx\yy_2$ is a string, a contradiction to $\yy_1\equiv_H\yy_2$. Hence our claim.

In view of the above claim, we may assume that $|\rho_r(\yy_1)||\rho_r(\yy_2)|>0$.

$(a)$ If $\gamma_1=\gamma_2$ but $\rho_r(\yy_1)\neq\rho_r(\yy_2)$. Since $\delta(\rho_r(\yy_1))=\delta(\rho_r(\yy_2))$, without loss assume that $\rho_r(\yy_2)$ is a proper right substring of $\rho_r(\yy_1)$. Thus if $\xx\rho_r(\yy_1)\in\rho\cup\rho^{-1}$ then $\xx\rho_r(\yy_2)$ is a string, a contradiction to $\yy_1\equiv_H\yy_2$. Thus $\rho_r(\yy_1)=\rho_r(\yy_2)$.

$(b)$ If $\gamma_1\neq\gamma_2$ then $t(\yy_1)=t(\yy_2)$ implies that $\theta(\gamma_1)=-\theta(\gamma_2)$. Since $\gamma_j^{-1}\yy_j$ is not a string and $\yy_1\equiv_H\yy_2$ we see that $\gamma_j^{-1}\yy_k$ is also not a string for $j\neq k$. 

\noindent{\textbf{Claim}:} $\gamma_j^{-1}\rho_r(\yy_k)$ is not a string.

Suppose for contradiction that $\gamma_j^{-1}\rho_r(\yy_k)$ is a string. Then there is a string $\xx_j$ such that $\xx_j\rho_r(\yy_j)\in\rho\cup\rho^{-1}$. Since $\delta(\xx_j)=\delta(\gamma_j)=-\delta(\gamma_k)$ we get that $\xx_j\yy_k$ is a string, a contradiction to $\yy_1\equiv_H\yy_2$. This completes the proof of the claim. In fact the same argument shows that whenever $\xx\yy_j$ is not a string then $\gamma_k^{-1}$ is the first syllable of $\xx$.

For the other direction, if $\rho_r(\yy_1)=\rho_r(\yy_2)$ then $\xx\yy_1$ is a string if and only if $\xx\yy_2$ is a string for any string $\xx$ with $s(\xx)=t(\yy_1)$ and thus $\yy_1\equiv_H\yy_2$. 

It only remains to verify the H-equivalence of $\yy_1$ and $\yy_2$ when $|\rho_r(\yy_1)||\rho_r(\yy_2)|>0$ and $\gamma_1\neq\gamma_2$. If $\xx\yy_j$ is a string then clearly $\gamma_j^{-1}$ is not the first syllable of $\xx$. Moreover the hypothesis says that $\gamma_k^{-1}$ is also not the first syllable of $\xx$ and hence $\xx\yy_k$ is also a string. This completes the proof that $\yy_1\equiv_H\yy_2$.   
\end{proof}

\begin{proposition}\label{bandHequivalence}
If $\xx,\yy$ are strings with $|\xx|>0$ such that $\xx\yy$ is a string and $\xx\yy\equiv_H\yy$ then $\xx$ is finite power of a cyclic permutation of a band.
\end{proposition}

\begin{proof}
It follows from the hypotheses that $\xx$ is a cyclic string.

Suppose $\delta(\xx)\neq0$. Then by the definition of a string algebra, let $n>1$ be minimal such that $\xx^n$ is not a string. Then for some $0\leq m\leq n-1$, $\xx^m\xx\yy$ is not a string but $\xx^m\yy$ is implying $\xx\yy\nequiv_H\yy$. Thus $\delta(\xx)=0$.

However if $\delta(\xx)=0$ but $\xx^2$ is not a string then the argument in the above paragraph shows that $\xx\yy\nequiv_H\yy$. Thus $\xx^2$ is a string. Therefore the conclusion follows from \cite[Remark~3.1.3]{GKS}.
\end{proof}

\begin{definition}
Say that a string $\yy$ is an \emph{H-string} if $\yy\nequiv_H\yy'$ for each proper left substring $\yy'$ of $\yy$. Further say that a string $\yy$ is a \emph{hereditary H-string} if each of its left substrings is an H-string.
\end{definition}

\begin{rmk}\label{leftsubhereditary}
A left substring of a hereditary string is also a hereditary string.
\end{rmk}

\begin{proposition}\label{hereditaryskeletal}
A hereditary H-string is skeletal.
\end{proposition}
\begin{proof}
Suppose $\yy$ is a string such that $N(\bb,\yy)>1$ for some band $\bb$. Let $\yy=\yy_2(\bb')^2\yy_1$ for some cyclic permutation $\bb'$ of $\bb$. Since $\delta(\bb')=0$ we have $\delta(\rho_r((\bb')^2\yy_1))\delta(\rho_r(\bb'\yy_1))\neq-1$ and $\rho_r((\bb')^2\yy_1)=\rho_r(\bb'\yy_1)$. Then Proposition \ref{newHequivcriterion} gives $(\bb')^2\yy_1\equiv_H\bb'\yy_1$, and hence $\yy$ is not a hereditary H-string. 
\end{proof}

The next result is an immediate consequence of Propositions \ref{skeletalfinitude} and \ref{hereditaryskeletal}.
\begin{corollary}\label{hereditaryHstringfinitude}
There are only finitely many hereditary H-strings.
\end{corollary}

Say that a string $\xx$ is a $1$-step H-reduction of $\yy$ if there is a partition $\yy=\yy_2\bb'\yy_1$, where $\bb'$ is a cyclic permutation of $\bb$, such that $\xx=\yy_2\yy_1$ and $\bb'\yy_1\equiv_H\yy_1$. For a string $\yy$ and $\bb\in\B(\yy)$ if $\Red\bb(\yy)$ exists and is a $1$-step H-reduction of $\yy$ then we denote it by $\HRed\bb(\yy)$. We further define $\hh_\bb(\yy):=\begin{cases}\HRed\bb(\yy)&\mbox{if }\HRed\bb(\yy)\mbox{ exists};\\\yy&\mbox{otherwise}.\end{cases}$

Say that a string $\xx$ is an H-reduction of $\yy$ if there is a finite sequence $\yy=\yy_0,\yy_1,\hdots,\yy_k=\xx$ of strings such that, for each $0<j\leq k$, $\yy_j$ is a $1$-step H-reduction of $\yy_{j-1}$.

\begin{rmk}\label{skeletonisHRed}
For a string $\yy$ and $\bb\in\B(\yy)$, Proposition \ref{newHequivcriterion} guarantees that if $N(\bb,\yy)>1$ then $\HRed\bb^{N(\bb,\yy)-1}(\yy)$ exists. Hence $\sk\yy$ is an H-reduction of $\yy$.
\end{rmk}

Say that a string $\yy$ is \emph{H-reduced} if it has no H-reduction.

\begin{rmk}
Consider the string algebra $\Lambda^{(vi)}$ from Figure \ref{notHreducedreduction} and the string $\yy:=JeHgFcaDB$ such that $\B(\yy)=\{\bb_1,\bb_2\}$, where $\bb_1:=aD$ and $\bb_2:=gFH$. Then $\yy$ is H-reduced but $\Red{\bb_1}(\yy)=JeHgFcB$ is not. Therefore a reduction of an H-reduced string is not necessarily H-reduced.
\begin{figure}[h]
    \centering
\begin{tikzcd}
    & v_3 \arrow[d, "d", bend left=49] \arrow[d, "a"', bend right=49] & v_6                                 & v_5 \arrow[l, "g"'] \arrow[ld, "f"] & v_8                &                    \\
v_1 & v_2 \arrow[l, "b"] \arrow[r, "c"']                              & v_4 \arrow[u, "h"] \arrow[rr, "e"'] &                                     & v_7 \arrow[u, "i"] & v_9 \arrow[l, "j"]
\end{tikzcd}
    \caption{$\Lambda^{(vi)}$ with $\rho=\{ba,cd,hc,ef,ij,ieca\}$}
    \label{notHreducedreduction}
\end{figure}
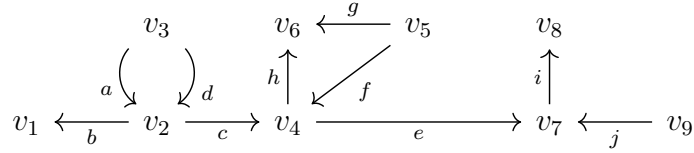
\end{rmk}

\begin{proposition}\label{Hreducedhereditary}
A string $\yy$ is H-reduced if and only if it is a hereditary H-string.
\end{proposition}
\begin{proof}
Suppose $\yy$ is not H-reduced then there is a partition $\yy=\yy_2\bb'\yy_1$, where $\bb'$ is a cyclic permutation of a band $\bb$, such that $\xx=\yy_2\yy_1$ and $\bb'\yy_1\equiv_H\yy_1$. Then $\bb'\yy$ is a left substring of $\yy$ which is not an H-string. Hence $\yy$ is not a hereditary H-string.

For the converse suppose $\yy$ is not a hereditary H-string. Then there is a partition $\yy=\yy'\xx''\xx'$ such that $\xx''\xx'\equiv_H\xx'$. Now Proposition \ref{bandHequivalence} guarantees that $\xx''=(\bb')^n$ for some $n\geq1$ and some cyclic permutation $\bb'$ of a band $\bb$. If $n>1$ then $\delta(\bb')=0$ together with Proposition \ref{newHequivcriterion} gives that $\bb'\xx'\equiv_H\xx'$. Hence for any $n\geq1$, $\yy'(\bb')^{n-1}\xx'$ is an H-reduction of $\yy$, thus implying that $\yy$ is not H-reduced.
\end{proof}

\begin{proposition}\label{notintersectingHequivbands}
Suppose $\zz$ is a string that is not a hereditary H-string. Further suppose that $\xx_1\yy_1$ and $\xx_2\yy_2$ are left substrings of $\zz$ such that $\yy_1\equiv_H\xx_1\yy_1$ and $\yy_2\equiv_H\xx_2\yy_2$. If $\xx_1$ and $\xx_2$ have a common substring of positive length then their union is a substring of a finite power of a band.
\end{proposition}
\begin{proof}
Using Proposition \ref{bandHequivalence} we know that $\xx_j$ is a finite power of a cyclic permutation $\bb'_j$ of band $\bb_j$, and then Proposition \ref{newHequivcriterion} gives that $\bb'_j\yy_j\equiv_H\yy_j$ for $j=1,2$. Without loss we may assume that neither $\xx_1$ is a substring of $\xx_2$ nor $\xx_2$ is a substring of $\xx_1$. Further without loss we assume that $\yy_1$ is a proper left substring of $\yy_2$. Since $\xx_1$ and $\xx_2$ have a common substring of positive length, $\yy_2$ is a left substring of $\xx_1\yy_1$, and the bands $\bb_1$ and $\bb_2$ also have a common syllable. Our aim is to show that $\bb_1=\bb_2$.

If not, then there exists an abnormal weak bridge $\bb_1\to\bb_2$. Let $\xx_1\yy_1=\zz'\yy_2$ for a string $\zz'$ of positive length. Since $\xx_1$ is a finite power of $\bb'_1$, $\bb'_1\xx_1\yy_1=\bb'_1\zz'\yy_2$ is a string. Since $\yy_2\equiv_H\bb'_2\yy_2$ we have that $\bb'_1\zz'\bb'_2\yy_2$ is a string and thus there is a path $\bb_2\to\bb_1$ in the bridge quiver, a contradiction to the domesticity of the algebra. Thus the proof.
\end{proof}

\begin{corollary}\label{reductioncommutativity}
Suppose $\zz$ is a skeletal string that is not a hereditary string. If $\bb_1,\bb_2$ are distinct bands such that $\zz_i:=\HRed{\bb_i}(\zz)$ exists for $i=1,2$. Then $\zz_{ji}:=\HRed{\bb_i}(\zz_j)$ exist for $i\neq j$, and $\zz_{12}=\zz_{21}$.
\end{corollary}

\begin{proof}
Since $\zz$ is skeletal, in view of Propositions \ref{bandHequivalence} and \ref{notintersectingHequivbands}, we have $\zz=\yy\bb'_2\yy_2=\yy\bb'_2\xx\bb'_1\yy_1$, where $\bb'_1,\bb'_2$ are cyclic permutations of $\bb_1,\bb_2$ respectively. For $i=1,2$, since $\HRed{\bb_i}(\zz)$ exists, we get $\bb'_i\yy_i\equiv_H\yy_i$. Clearly $\zz_1=\yy\bb'_2\xx\yy_1$ and $\zz_2=\yy\xx\bb'_1\yy_1$. 

Since $\bb'_1\yy_1\equiv_H\yy_1$ and $\bb'_2\yy_2\equiv_H\yy_2$, for any string $\xx'$, $\xx'\bb'_2\yy_2$ is a string if and only if $\xx'\yy_2=\xx'\xx\bb'_1\yy_1$ is a string if and only if $\xx'\xx\yy_1$ is a string. Thus $\zz_{12}:=\HRed{\bb_2}(\zz_1)$ exists. On the other hand, since $\bb'_2$ does not interfere with $\bb_1$-reduction of $\zz$ as well as $\zz_2$, clearly $\zz_{21}:=\HRed{\bb_1}(\zz_2)$ exists and $\zz_{12}=\yy\xx\yy_1=\zz_{21}$.
\end{proof}


In fact the above result can be generalized even when certain H-reductions do not exist.
\begin{theorem}\label{Hcompassociativity}
Suppose $\zz$ is a string and $\bb_1,\bb_2\in\B(\zz)$. Then $$\hh_{\bb_1}(\hh_{\bb_2}(\zz))=\hh_{\bb_2}(\hh_{\bb_1}(\zz)).$$
\end{theorem}

\begin{proof}
In view of Remark \ref{skeletonisHRed} we may assume that $\zz$ is skeletal. Suppose $\zz=\yy_3\bar\bb_2\yy_2\bar\bb_1\yy_1$ for some cyclic permutations $\bar\bb_1,\bar\bb_2$ of bands $\bb_1,\bb_2$ respectively. There are four cases.

\noindent{\textbf{Case I:}} If $\HRed{\bb_j}(\zz)$ exists for $j=1,2$ then the conclusion follows from Corollary \ref{reductioncommutativity}. 

\noindent{\textbf{Case II:}} If $\HRed{\bb_j}(\zz)$ does not exist for $j=1,2$ then both sides of the required identity equal $\zz$, and hence the conclusion.

\noindent{\textbf{Case III:}} If $\zz_2:=\HRed{\bb_2}(\zz)$ exists but $\HRed{\bb_1}(\zz)$ does not exist then it is enough to argue that $\HRed{\bb_1}(\zz_2)$ does not exist.

Let $\zz'$ be the longest left substring of $\zz$ such that $\bb'_1\zz'$ is a left substring of $\zz$ for some cyclic permutation $\bb'_1$ of $\bb_1$. Let $\zz''$ be the maximal common left substring of $\zz$ and $\zz_2$. Then $\bb'_2\zz''\equiv_H\zz''$ for a cyclic permutation $\bb'_2$ of $\bb_2$. Moreover $\zz'$ is a left substring of $\zz''$.

If $\bb'_1\zz'$ is not a left substring of $\zz_2$ then $\bb'_1\zz'=\bar\zz\zz''$, where $|\bar\zz|>0$. Since $\bar\zz$ is a right substring of $\bb'_1$ and a left substring of $\bb'_2$ there is an abnormal weak bridge $\bb_1\xrightarrow{\uu}\bb_2$ such that $\bar\zz$ is a substring of $\uu^c$. Maximality of $\zz'$ gives that $\bar\zz$ is a right substring of $\uu^c$. However $\uu^\beta\bb'_1\zz'=\uu^\beta\bar\zz\zz''$ is a string but $\uu^\beta\bar\zz\bb'_2\zz''$ is not a string, a contradiction to $\bb'_2\zz''\equiv_H\zz''$. Hence $\bb'_1\zz'$ is a left substring of $\zz_2$.

\textbf{Claim:} $\zz'$ is the longest left substring $\ww$ of $\zz_2$ such that $\bb''_1\ww$ is a left substring of $\zz_2$ for some cyclic permutation $\bb''_1$ of $\bb_1$.

Let $\zz=\zz_3\bb'_2\zz'',\zz_2=\zz_3\zz''$. Suppose $\bb''_1\tilde\zz\zz'$ is a left substring of $\zz_2$ for some cyclic permutation $\bb''_1$ of $\bb_1$ and some string $\tilde\zz$ with $|\tilde\zz|>0$. Then maximality of $\zz'$ ensures that $\zz''=\bb'_1\zz'$. However clearly $\bb'_2\bb'_1\zz'\nequiv_H\bb'_1\zz'$. This contradiction completes the proof of the claim, and hence of this case.

\noindent{\textbf{Case IV:}} If $\zz_1:=\HRed{\bb_1}(\zz)$ exists but $\HRed{\bb_2}(\zz)$ does not exist then it is enough to argue that $\HRed{\bb_2}(\zz_1)$ does not exist.

Suppose for contradiction that $\HRed{\bb_2}(\zz_1)$ exists. Then $N(\bb_2,\zz_1)=1$, and there is a shortest left substring $\zz''$ of $\zz_1$ such that $\bb'_2\zz''$ is a left substring of $\zz_1$ for some cyclic permutation $\bb'_2$ of $\bb_2$ and $\bb'_2\zz''\equiv_H\zz''$.

Let $\zz'$ be the maximal common left substring of $\zz$ and $\zz_1$. Then $\bb'_1\zz'$ is a left substring of $\zz$ for some cyclic permutation $\bb'_1$ of $\bb_1$ for which $\bb'_1\zz'\equiv_H\zz'$.

If $\zz'=\tilde\zz\zz''$ then since $\bb'_2\zz''\equiv_H\zz''$, $\tilde\zz\bb'_2\zz''$ is a string.
Then Proposition \ref{newExtHequivcriterion} gives $\tilde\zz\bb'_2\zz''\equiv_H\tilde\zz\zz''(=\zz')$. Since $\zz'\equiv_H\bb'_1\zz'$, $\bb'_1\tilde\zz\bb'_2\zz''$ is a string. Then by \cite[Lemma~3.3.4]{GKS} there is a path in the bridge quiver from $\bb_2$ to $\bb_1$. However since $\bar\bb_2\yy_2\bar\bb_1$ is a string we already have a path from $\bb_1$ to $\bb_2$. This is a contradiction to domesticity. Hence there is a string $\zz'''$ such that $\zz''=\zz'''\zz'$.

Since a $1$-step $\bb_2$-reduction of $\zz$ is not an H-reduction, we have $\bb'_2\zz'''\bb'_1\zz'\nequiv_H\zz'''\bb'_1\zz'$, and hence there is a string $\xx$ such that exactly one of $\xx\bb'_2\zz'''\bb'_1\zz'$ and $\xx\zz'''\bb'_1\zz'$ is a string. If the former is a string then using $\delta(\bb'_2)=0$ together with the H-equivalences $\bb'_2\zz'''\zz'\equiv_H\zz'''\zz'$ and $\bb'_1\zz'\equiv_H\zz'$ it can be easily shown that the latter is also a string, a contradiction. Similarly if the latter is a string then we can argue that the former is also a string, again a contradiction. Hence our assumption is wrong, and the conclusion follows.
\end{proof}

The above result shows that every string $\zz$ has a unique H-reduced iterated H-reduction, which we denote by $\hh(\zz)$.

\section{The weak arch bridge quiver}


\begin{definition}
If $\Pp:=(\uu_1,\uu_2,\hdots,\uu_n)$ is a path in $\overline\Q^{\mathrm{Ba}}$ with $n\geq1$ then define $\hh(\Pp)$ by $$\tbq(\Pp)\hh(\Pp)\sbq(\Pp)=\hh_{\tbq(\uu_1)}\hdots\hh_{\tbq(\uu_{n-1})}(\sk{\tbq(\uu_n)\uu_n\hdots\tbq(\uu_1)\uu_1\sbq(\uu_1)}).$$
\end{definition}




If $\Pp$ is a path in $\overline\Q^{\mathrm{Ba}}$ with $\sbq(\Pp)=\bb_1$ and $\tbq(\Pp)=\bb_2$ then say that $\bb_1\xrightarrow{\hh(\Pp)}\bb_2$ is a \emph{weak arch bridge}.

Note that $\bb_2\uu\bb_1$ is a skeletal string when $\bb_1\xrightarrow{\uu}\bb_2$ is a weak arch bridge.


The weak arch bridge quiver, denoted $\bHQ$, has as its vertex set the set $Q_0^{\mathrm{Ba}}$, and as its arrows the weak arch bridges between bands. We continue to use the notations $\sbq(\uu)$ and $\tbq(\uu)$ to denote the source and the target of an arrow $\uu$ in $\bHQ$.

\begin{rmk}\label{weakbridgesareweakarchbridges}
It follows from Corollary \ref{hereditaryHstringfinitude} and Proposition \ref{Hreducedhereditary} that $\bHQ$ is a finite quiver that contains $\overline{\Q}^{\mathrm{Ba}}$ as a subquiver.
\end{rmk}

\begin{definition}
If $\Pp_1$ and $\Pp_2$ are two paths in $\overline\Q^{\mathrm{Ba}}$ such that $\sbq(\Pp_2)=\tbq(\Pp_2)$ then define $$\hh(\Pp_2)\ch\hh(\Pp_1):=\hh(\Pp_1+\Pp_2).$$
\end{definition}

The next two key results show that an appropriate subset of the (finite) set of H-reduced strings is precisely the set of weak arch bridges.
\begin{theorem}\label{weakarchbridgeHreduced}
A weak arch bridge is an H-reduced string.
\end{theorem}

\begin{proof}
Suppose $\uu$ is a weak arch bridge. Then $\uu=\hh(\Pp)$ for some path $\Pp$ in $\overline{\Q}^{\mathrm{Ba}}$. We use induction on the length of $\Pp$ to prove this result.

If $|\Pp|=1$ then $\uu$ is a weak bridge which is band-free and hence an H-reduced string.

Now suppose $|\Pp|>1$. Then $\uu=\uu_2\ch\hh(\Pp_1)$, where $\Pp=\Pp_1+(\uu_2)$. Since $|\Pp_1|<|\Pp|$, the induction hypothesis gives that $\uu_1:=\hh(\Pp_1)$ is an H-reduced string. Then $\uu=\uu_2\ch\uu_1$.

Suppose $\uu$ is not H-reduced. Then there is some band $\bb_1$ different from $\bb:=\tbq(\Pp_1)$ such that $\HRed{\bb_1}(\uu)$ exists. Let $\uu''_1$ be the maximal left substring of $\uu$ such that $\bb'_1\uu''_1$ is a left substring of $\uu$ and $\bb'_1\uu''_1\equiv_H\uu''_1$ for some cyclic permutation $\bb'_1$ of $\bb_1$.





There are four cases.

\noindent{\textbf{Case I:}} $N(\bb,\uu_2\uu_1)=1$ and $N(\bb,\uu)=1$. 

Since $N(\bb,\uu_1)=N(\bb,\uu_2)=0$ but $N(\bb,\uu)=1$, $\delta(\gamma\gamma')=0$, where $\gamma$ is the first syllable of $\uu_2$ and $\gamma'$ is the last syllable of $\uu_1$. Clearly $\gamma\gamma'$ is a substring of $\bb$. Let $\uu'_1$ be the substring of $\uu_1$ such that $\bb'\uu'_1$ is a string for some cyclic permutation $\bb'$ of $\bb$.

Since $\uu_1$ is H-reduced, $\bb'_1\uu''_1$ is not a left substring of $\uu_1$. If $\uu''_1=\zz\uu_1$ for some string $\zz$ with $|\zz|>0$ then it can be easily concluded from $\delta(\gamma\gamma')=0$ that $\bb'_1\zz\equiv_H\zz$, a contradiction to the fact that $\uu_2$ is H-reduced.

If $\uu''_1$ is a proper left substring of $\uu_1$ then some cyclic permutations of $\bb$ and $\bb'$ have $\gamma\gamma'$ as a common substring. Since $\delta(\gamma\gamma')=0$ we conclude using domesticity of the string algebra that $\bb=\bb_1$, a contradiction. Hence $\uu''_1=\uu_1$.

Now $\bb'=\zz_2\zz_1$ and $\bb'_1=\zz_3\zz_2$, where $\uu_1=\zz_1\uu'_1$, $\uu_2=\uu'_2\zz_2$ and $\delta(\zz_2)\neq0$. Then $\zz_3\zz_2\zz_1\uu'_1$ is a string. Since $\delta(\zz_2\zz_1)=0$, $\bb'_1\bb=\zz_3\zz_2\zz_1\zz_2$ is a string.

On the other hand, since $\bb\uu_1$ is a string and $\uu_1=\uu''_1\equiv_H\bb'_1\uu_1$, we get that $\bb\bb'_1\uu_1$, and hence $\bb\bb'_1$ are strings, a contradiction to domesticity.

\noindent{\textbf{Case II:}}
$N(\bb,\uu_2\uu_1)=0$ and $N(\bb,\uu)=1$.

Let $\gamma$ (resp. $\gamma'$) be the first (resp. last) syllable of $\uu_2$ (resp. $\uu_1$). Let $\gamma(\bb)$ and $\gamma(\bb')$ be the first and the last syllables of $\bb$. Then $\theta(\gamma(\bb))=1=-\theta(\gamma'(\bb))$.

Suppose $\uu_1=\yy\uu''_1$ with $|\yy|>0$. Since $\uu_1$ is H-reduced, $\bb'_1\uu''_1$ is not a left substring of $\uu_1$. Hence $\bb'_1=\yy_1\yy$ for a positive length string $\yy_1$. In view of \cite[Corollary~3.4.1]{GKS} we have $\bb=\yy_2\yy_1$ for a positive length string $\yy_2$. Hence there is an abnormal weak bridge $\bb_1\xrightarrow{\uu_0}\bb$ such that $\yy_1$ is a substring of $\uu_0^c$. Since $\gamma(\bb)$ is the first syllable of $\yy_1$ we have $\delta(\yy_1)=1$.

By maximality of $\uu''_1$ we see that $\yy_1$ is a right substring of $\uu_0^c$. Now $\yy_1=\uu_0^c$ if and only if $\gamma'\neq\gamma'(\bb)$. To see that latter condition holds, we indeed have $\delta(\gamma'(\bb))=-1$ and $\delta(\yy_1)=-1=-\delta(\gamma')$. Hence $\yy_1=\uu_0^c$. But then $\alpha(\uu_0)=\gamma'$ which gives $\delta(\alpha(\uu_0))=\delta(\gamma')=\delta(\uu_0^c)$, a contradiction to Proposition \ref{constantsign} since $|\uu_0^c|>0$.

If $\uu''_1=\zz\uu_1$ for some string $\zz$ then \cite[Corollary~3.4.1]{GKS} ensures that $|\zz|>0$ and that $\bb'_1\zz\uu_1$ is not a substring of $\bb\uu_1$.

If $\bb=\zz'\zz$ for some $\zz'$ then $\bb'_1=\zz''\zz'$ for some $\zz''$ with $|\zz''|>0$. In particular there is an abnormal weak bridge $\bb\to\bb_1$.

On the other hand since $\bb'_1\uu''_1\equiv_H\uu''_1$ and $\bb\zz'\uu''_1$ is a string, so is $\bb\zz'\bb'_1\uu''_1$. Hence there is a path from $\bb_1$ to $\bb$ in the bridge quiver.

The conclusions of the above two paragraphs together contradict the domesticity of algebra.

Therefore $\uu''_1=\bar\zz\bb\uu_1$ with $|\bar\zz|>0$. Since $\uu_2$ is H-reduced we have $\bb'_1\bar\zz\nequiv_H\bar\zz$. Thus there is a string $\bar\xx$ with $\delta(\bar\xx)\neq0$ such that exactly one of $\bar\xx\bb'_1\bar\zz$ and $\bar\xx\bar\zz$ is a string.

If $\bar\xx\bb'_1\bar\zz$ is a string then so is $\bar\xx\bb'_1\bar\zz\bb\uu_1$. Further using $\bb'_1\bar\zz\bb\uu_1\equiv_H\bar\zz\bb\uu_1$ we see that $\bar\xx\bar\zz\bb\uu_1$, and hence $\bar\xx\bar\zz$, are strings too, a contradiction to the above paragraph. Therefore $\bar\xx\bar\zz$ is a string but $\bar\xx\bb'_1\bar\zz$ is not. Once again using $\bb'_1\bar\zz\bb\uu_1\equiv_H\bar\zz\bb\uu_1$ we see that $\bar\xx\bar\zz\bb\uu_1$ is not a string. Since $\delta(\bb)=0$, we conclude that $\bar\xx\bar\zz\bb$ is not a string.

Let $\yy_0$ be the shortest right substring of $\bb$ such that $\bar\xx\bar\zz\yy_0$ is not a string. Since $\bar\xx\bar\zz$ is a string we get $|\yy_0|>0$. Thus $\gamma'(\bb)\in\yy_0$, and hence $\delta(\bar\xx\bar\zz\yy_0)=-1$. As a consequence, we see that $\gamma\neq\gamma(\bb)$.

Since $\bar\xx\bar\zz\bb\uu_1$ is not a string, $\bar\xx\bb'_1\bar\zz\bb\uu_1$ is not either. Since $\delta(\bb'_1)=0$ we get that $\bar\xx\bb'_1$ is not a string. Since no two elements of $\rho$ are comparable, we get that $\bar\xx\bar\zz\yy_0$ is right subword of $\bar\xx\bb'_1$. In particular, $\yy_0$ is a common substring of bands $\bb$ and $\bb_1$. Hence there is an abnormal weak bridge $\bb\xrightarrow{\tilde\uu}\bb_1$ such that $|\tilde\uu^c|>0$. As a consequence the first syllable of $\bar\zz$, i.e., $\gamma$, is the exit syllable of $\tilde\uu$. But from the above paragraph we see that $\delta(\beta(\tilde\uu))=\delta(\gamma)=-1=\delta(\yy_0)=\delta(\tilde\uu^c)$, a contradiction to $|\tilde\uu^c|>0$ by Proposition \ref{constantsign}.

Therefore in each subcase we obtained a contradiction. This completes the proof of Case II.

\noindent{\textbf{Case III:}} $N(\bb,\uu_2\uu_1)=1$ and $N(\bb,\uu)=0$. 

As in Case I since $N(\bb,\uu_1)=N(\bb,\uu_2)=0$ but $N(\bb,\uu_2\uu_1)=1$, $\delta(\gamma\gamma')=0$, where $\gamma$ is the first syllable of $\uu_2$ and $\gamma'$ is the last syllable of $\uu_1$. Clearly $\gamma\gamma'$ is a substring of $\bb$. Suppose $\uu'_1$ is the maximal left substring of $\uu$ for which $\bb'\uu'_1\equiv_H\uu'_1$ for some cyclic permutation $\bb'$ of $\bb$. Let $\uu_2\uu_1=\uu'_2\bb'\uu'_1$. Then $\uu_2\ch\uu_1=\HRed\bb(\uu_2\uu_1)=\uu'_2\uu'_1$. Let $\uu'_2\uu'_1=\uu''_2\bb'_1\uu''_1$.

Since $\uu_1$, and hence $\uu'_1$, is H-reduced, $\bb'_1\uu''_1$ is not a left substring of $\uu'_1$. So there are two possibilities.

If $\uu''_1=\zz\uu'_1$ then for any string $\zz'$, $\zz'\zz\bb'\uu'_1$ is a string if and only if $\zz'\zz\uu'_1$ is a string if and only if $\zz'\bb'_1\zz\uu'_1$ is a string if and only if $\zz'\bb'_1\zz\bb'\uu'_1$ is a string. Therefore as $\delta(\gamma\gamma')=0$ we have $\bb'_1\tilde\zz\equiv_H\tilde\zz$, where $\tilde\zz$ is the right substring of $\zz\bb'\uu'_1$ that is also a left substring of $\uu_2$. This is a contradiction to the fact that $\uu_2$ is H-reduced.

On the other hand if $\uu'_1=\zz\uu''_1$ for some string $\zz$ with $|\zz|>0$ then let $\bb'_1=\zz_1\zz$. Since $\zz_1\zz\uu''_1$ is a string, so are $\zz_1\bb'\zz\uu''_1$ and $\zz_1\bb'\zz\bb'_1\uu''_1$. Thus there is an abnormal weak bridge $\bb_1\xrightarrow{\bar\uu}\bb$. By maximality of $\uu'_1$, we see that $|\zz_1|<|\bar\uu^\beta|$. Moreover $\bar\uu^\beta\bb'$ is not a string, and hence $\bar\uu^\beta\bb'\zz\uu''_1$ is not too.

However $\bar\uu^\beta\zz\uu''_1=\bar\uu^\beta\uu'_1$ is a string. Since $\bb'\uu'_1\equiv_H\uu'_1$, $\bar\uu^\beta\bb'\zz\uu''_1$ is a string, a contradiction to the paragraph above. This completes the proof of Case III.

\noindent{\textbf{Case IV:}} 
$N(\bb,\uu_2\uu_1)=0$ and $N(\bb,\uu)=0$. 

The proof of this case is similar to the proof of Case III with appropriate modifications.

Thus the proof of the theorem is complete.
\end{proof}


The following result is the converse to Theorem \ref{weakarchbridgeHreduced}.
\begin{proposition}\label{Hreducedweakarchbridge}
If $\bb,\bb'$ are bands and $\yy$ is an H-reduced string such that $\bb'\yy\bb$ is a skeletal string then $\yy$ is a weak arch bridge.
\end{proposition}

\begin{proof}
Suppose $\B(\yy)=\{\bb_1,\bb_2,\hdots,\bb_n\}$ where $j\leq k$ if and only if there is a possibly trivial path from $\bb_j$ to $\bb_k$ in the weak bridge quiver. Using $\bb_0:=\bb$ and $\bb_{n+1}:=\bb'$ let $\bb_j\xrightarrow{\yy_j}\bb_{j+1}$ be a weak bridge that is a substring of $\bb'\yy\bb$ for $0\leq j\leq n$.

We use induction on $n=|\B(\yy)|$ to prove the result.

If $n=0$ then $\yy$ is a weak bridge and hence a weak arch bridge by Remark \ref{weakbridgesareweakarchbridges}.

If $n>0$ then let $\zz_n$ be the shortest left substring of $\yy$ such that $\bb_n\zz_n\bb$ is a string. Since $\yy$ is H-reduced we get that $\yy=\yy_n\ch\zz_n$. Then $|\B(\zz_n)|<n$. Hence by induction hypothesis there is some path $\Pp=(\uu_1,\hdots,\uu_p)$ in the weak bridge quiver such that $\zz_n=\hh(\Pp)$. Since $\yy_n$ is a weak bridge and $\yy=\yy_n\ch\zz_n$, it is immediate that $\yy=\hh(\Pp+(\yy_n))$, and hence $\yy$ is a weak arch bridge.
\end{proof}

\section{The arch bridge quiver}\label{secarchbrquiv}
As bridges were defined to be the $\circ$-irreducible weak bridges we define ``arch bridges'' to be $\ch$-irreducible weak arch bridges, and study the properties of their H-compositions.
\begin{definition}
Say that a weak arch bridge $\bb_1\xrightarrow{\uu}\bb_2$ is an \emph{arch bridge} if $\uu\neq\hh(\Pp)$ for any path $\Pp$ of length at least $2$ from $\bb_1$ to $\bb_2$ in $\bHQ$.
\end{definition}

The arch bridge quiver, denoted $\HQ$, is the subquiver of the weak arch bridge quiver $\bHQ$ consisting of only arch bridges.


\begin{rmk}\label{archbridgeisweakbridge}

If a weak arch bridge $\bb_1\xrightarrow{\hh(\Pp)}\bb_2$ is an arch bridge for some path $\Pp$ from $\bb_1$ to $\bb_2$ in the weak bridge quiver, then from the inductive construction of $\hh$ we conclude that $|\Pp|=1$. As a consequence $\HQ$ is a subquiver of $\overline\Q^{\mathrm{Ba}}$.

\end{rmk}

\begin{rmk}\label{bridgeisarchbridge}
If $\uu\in\overline{\Q}^{\mathrm{Ba}}$ factors as $\uu=\uu_2\ch\uu_1$ then $\uu_2\ch\uu_1=\uu_2\circ\uu_1$. Hence a bridge is an arch bridge and $\Q^{\mathrm{Ba}}$ is a subquiver of $\HQ$.
\end{rmk}

In \S\ref{newextsemibrquiv} we will study a class of arch bridges containing the class of bridges.
\begin{proposition}\label{pathforweakarchbridge}
Suppose $\uu$ is a weak arch bridge. Then there is some path $\Pp$ in $\HQ$ such that $\uu=\hh(\Pp)$.
\end{proposition}

\begin{proof}
Suppose $\sbq(\uu)\xrightarrow{\uu}\tbq(\uu)$ is a weak arch bridge. Then $\uu$ is H-reduced by Theorem \ref{weakarchbridgeHreduced}. Let $\B(\uu)=\{\bb_1,\bb_2,\hdots,\bb_n\}$ where, there is a path from $\bb_i$ to $\bb_{i+1}$ in $\overline{\Q}^{\mathrm{Ba}}$ for $1\leq i<n$. Choosing $\bb_0:=\sbq(\uu)$ and $\bb_{n+1}:=\tbq(\uu)$, there are weak bridges $\bb_i\xrightarrow{\uu_i}\bb_{i+1}$ such that $$\tbq(\uu)\uu\sbq(\uu)=\hh_{\bb_n}\hdots\hh_{\bb_1}(\sk{\tbq(\uu)\uu_n\bb_n\hdots\uu_2\bb_2\uu_1\bb_1\uu_0\sbq(\uu)}).$$

This shows that it is enough to prove the result when $\uu$ is a weak bridge--the longer path can be obtained as a concatenation.

For non-trivial weak bridges $\uu,\uu',\uu''$ with $\uu=\uu'\circ\uu''$, say that $\uu$ is more complex than $\uu'$ and $\uu''$. We use `more complex' as the transitive closure of such a relation. Since the weak bridge quiver is finite, the complexity order is so too. We use induction on the complexity of a weak bridge to prove the result.


Note that $\uu\in\HQ$ if and only if whenever $\uu=\uu_2\circ\uu_1$ then $\uu\neq\uu_2\ch\uu_1$. 

Hence if $\uu\in\HQ$ we have $\uu=\hh((\uu))$. 

On the other hand if $\uu\notin\HQ$ then there is a non-trivial factorization $\uu=\uu_2\ch\uu_1$. Using the induction hypothesis for the weak bridges $\uu_i$ we obtain paths $\Pp_i$ in $\HQ$ such that $\uu_i=\hh(\Pp_i)$ for $i=1,2$ so that $\uu=\hh(\Pp_1+\Pp_2)$.
\end{proof}

\begin{theorem}\label{hisinjective}
Suppose $\Pp,\Pp'$ are paths in $\HQ$ with the same source and target such that $\hh(\Pp)=\hh(\Pp')$ then $\Pp=\Pp'$.
\end{theorem}

\begin{proof}
Let $\yy:=\hh(\Pp)=\hh(\Pp')$ and $\bb_1:=\sbq(\Pp)=\sbq(\Pp')$. If $\bb\in\B(\yy)$ then there are non-trivial partitions $\Pp=\Pp_1+\Pp_2$ and $\Pp'=\Pp'_1+\Pp'_2$ such that $\bb=\tbq(\Pp_1)=\tbq(\Pp'_1)$ and $\hh(\Pp_j)=\hh(\Pp'_j)$ for $j=1,2$. Hence we may assume that $\yy$ is band-free.

Let $n:=|\Pp|,m:=|\Pp'|$, $\Pp=\tilde\Pp+(\uu_n)$ and $\Pp'=\tilde\Pp'+(\uu'_m)$. We use induction on $\min\{n,m\}$ to prove the result.

For the base case, if $\min\{n,m\}=1$ then the conclusion follows by the definition of an arch bridge.

For the inductive case assume that $\min\{n,m\}>1$. Theorem \ref{Hcompassociativity} gives that $\yy=\uu_n\ch\hh(\tilde\Pp)=\uu'_m\ch\hh(\tilde\Pp')$. Suppose $\bb:=\sbq(\uu_n)$ and $\bb':=\sbq(\uu'_m)$.

\textbf{Claim:} $\bb=\bb'$

Suppose not. Since $\yy=\uu_n\ch\hh(\tilde\Pp)$ there is a shortest left substring $\xx$ of $\yy$ and a cyclic permutation $\tilde\bb$ of $\bb$ such that $\tilde\bb\xx\bb_1\equiv_H\xx\bb_1$. Similarly since $\yy=\uu'_m\ch\hh(\tilde\Pp')$ there is a shortest left substring $\xx'$ of $\yy$ and a cyclic permutation $\tilde\bb'$ such that $\tilde\bb'\xx'\bb_1\equiv_H\xx'\bb_1$. Suppose $\yy=\ww\xx=\ww'\xx'$. Without loss we may assume that $\xx'=\zz\xx$ for some (possibly length $0$) string $\zz$ and $\bar\yy:=\ww'\tilde\bb'\zz\tilde\bb\xx$ is a word.

Since $\tilde\bb\xx\bb_1\equiv_H\xx\bb_1$ and $\tilde\bb'\zz\xx\bb_1$ is a string, we get that $\bar\yy=\tilde\bb'\zz\tilde\bb\xx\bb_1$ is a string. In particular, there is a weak bridge $\bb\xrightarrow{\uu}\bb'$ such that $\uu_n=\uu'_m\circ\uu$.

Using the two H-equivalences above, for a string $\zz'$ such that $s(\zz')=t(\zz)$, we have $\zz'\tilde\bb'\zz\tilde\bb\xx\bb_1$ is a string if and only if $\zz'\tilde\bb'\zz\xx\bb_1$ is a string if and only if $\zz'\zz\xx\bb_1$ is a string if and only if $\zz'\zz\tilde\bb\xx\bb_1$ is a string. Since $\delta(\tilde\bb)=0$, we see that $\zz\tilde\bb\equiv_H\tilde\bb'\zz\tilde\bb$. Hence $\uu_n=\uu'_m\ch\uu$, a contradiction to $\uu_n\in\HQ$. This proves the claim.

It now follows from the claim that $\uu_n=\uu'_m$, and hence $\hh(\tilde\Pp)=\hh(\tilde\Pp')$. Since $\min\{|\tilde\Pp|,|\tilde\Pp'|\}<\min\{n,m\}$ we get $\tilde\Pp=\tilde\Pp'$ by the induction hypothesis. As a consequence we get $\Pp=\Pp'$ as required.
\end{proof}

The next result is in stark contrast with the composition $\circ$ of bridges in view of Example \ref{nonassociativityweakbridge}--its proof easily follows from Propositions \ref{pathforweakarchbridge}, \ref{hisinjective}, and Theorem \ref{Hcompassociativity}.

\begin{theorem}\label{Hassociativity}
Suppose $\bb_0\xrightarrow{\uu_1}\bb_1\xrightarrow{\uu_2}\bb_2\xrightarrow{\uu_3}\bb_3$ is a sequence of non-trivial weak arch bridges. Then $\uu_3\ch(\uu_2\ch\uu_1)=(\uu_3\ch\uu_2)\ch\uu_1$.
\end{theorem}

\section{The extended arch bridge quiver}\label{secextarchbrquiv}
Fix a string $\xx_0$ and $i\in\{1,-1\}$. The goal of this short section is to document versions of the definitions as well as results in the above sections ``relative to the pair $(i,\xx_0)$''. The proofs of the relative versions are omitted as they are similar to the proofs of the absolute versions with appropriate modifications. 

Say that $\yy\in H_l^i(\xx_0)$ is a hereditary H-string relative to $(i,\xx_0)$ if
\begin{itemize}
    \item for any distinct left substrings $\yy_1,\yy_2$ of $\yy$ in $H_l^i(\xx_0)$, different from $\xx_0$, we have $\yy_1\nequiv_H\yy_2$;
    \item for any left substring $\yy_1$ of $\yy$ in $H_l^i(\xx_0)$, different from $\xx_0$, if $\yy_1\equiv_H\xx_0$ and $\yy=\yy'_1\yy_1$ with $|\yy'_1|>0$ then $\theta(\yy'_1)=-i$.
\end{itemize}

If $\yy\in H_l^i(\xx_0)$, $\bb\in\B(\xx_0;\yy)$ then a $1$-step $\bb$-reduction of $\yy$ relative to $(i,\xx_0)$ exists if $\Red\bb(\yy)$ exists and is a string in $H^i_l(\xx_0)$, and we denote it by $\Red\bb^{(i,\xx_0)}(\yy)$. If $\Red\bb^{(i,\xx_0)}(\yy)$ is an H-reduction of $\yy$ then we denote it by $\HRed\bb^{(i,\xx_0)}(\yy)$. We further define $$\hh_\bb^{(i,\xx_0)}(\yy):=\begin{cases}\HRed\bb^{(i,\xx_0)}(\yy)&\mbox{if }\HRed\bb^{(i,\xx_0)}(\yy)\mbox{  exists};\\\yy&\mbox{otherwise}.\end{cases}$$ With appropriate modifications in the definition of an H-reduced string to get its version relative to $(i,\xx_0)$ we get the following relative version of Proposition \ref{Hreducedhereditary}.
\begin{proposition}
A string $\yy\in H_l^i(\xx_0)$ is H-reduced relative to $(i,\xx_0)$ if and only if it is a hereditary H-string relative to $(i,\xx_0)$.
\end{proposition}


The following extension of Theorem \ref{Hcompassociativity} gives that every string $\zz\in H_l^i(\xx_0)$ has a unique iterated H-reduction that is H-reduced with respect to $(i,\xx_0)$.
\begin{theorem}
Suppose $\zz\in H_l^i(\xx_0)$ and $\bb_1,\bb_2\in\B(\zz)$. Then $$\hh_{\bb_1}^{(i,\xx_0)}(\hh_{\bb_2}^{(i,\xx_0)}(\zz))=\hh_{\bb_2}^{(i,\xx_0)}(\hh_{\bb_1}^{(i,\xx_0)}(\zz)).$$
\end{theorem}

Now we define the operator $\hh(\xx_0;\mbox{-})$ to all paths in $\overline\Q^{\mathrm{Ba}}_i(\xx_0)$ starting at $\xx_0$; we omit the reference to $i$ as it will always be clear from the context.

\begin{definition}
We associate to each path $\Pp:=(\uu_1,\hdots,\uu_m)$ in $\overline{\Q}^{\mathrm{Ba}}_i(\xx_0)$ with $m\geq1$ and $\sbq(\Pp)=\xx_0$ a string $\hh(\xx_0;\Pp)$ inductively as follows.

$$\tbq(\Pp)\hh(\xx_0;\Pp)\xx_0=\hh_{\tbq(\uu_1)}^{(i,\xx_0)}\hdots\hh_{\tbq(\uu_{n-1})}^{(i,\xx_0)}(\sk{\tbq(\uu_n)\uu_n\hdots\tbq(\uu_1)\uu_1\xx_0}).$$

We further define, by dropping the reference to $\xx_0$ when it is clear from the context, the string $\hh(\Pp)$ by $\hh(\Pp)\xx_0=\hh(\xx_0;\Pp)$.

Moreover, if $\Pp_1$ and $\Pp_2$ are two paths such that $\Pp_1+\Pp_2$ exists in $\overline{\Q}^{\mathrm{Ba}}_i(\xx_0)$ and $\sbq(\Pp_1)=\xx_0$ then define $$\hh(\Pp_2)\ch^{(i,\xx_0)}\hh(\Pp_1):=\hh(\Pp_1+\Pp_2).$$
\end{definition}

For each path $\Pp:=(\uu_1,\hdots,\uu_m)$ in $\overline{\Q}^{\mathrm{Ba}}_i(\xx_0)$ with $m\geq1$, say that $\sbq(\Pp)\xrightarrow{\hh(\Pp)}\tbq(\Pp)$ is
\begin{itemize}
    \item a \emph{weak arch bridge} if $\sbq(\Pp)$ and $\tbq(\Pp)$ are bands;
    \item a \emph{weak half $i$-arch bridge} if $\sbq(\Pp)=\xx_0$ and $\tbq(\Pp)$ is a band;
    \item a \emph{weak zero $i$-arch bridge} if $\sbq(\Pp)=\xx_0$ and $\tbq(\Pp)$ is a $0$-length string;
    \item a \emph{weak reverse half arch bridge} if $\sbq(\Pp)$ is a band and $\tbq(\Pp)$ is a $0$-length string.
\end{itemize}
The extended weak arch bridge quiver, denoted $\bHQ_i(\xx_0)$, is the quiver whose vertices are the same as the vertices of $\overline{\Q}^{\mathrm{Ba}}_i(\xx_0)$ and whose arrows are weak arch bridges, weak half $i$-arch bridges, weak zero $i$-arch bridges, and weak reverse half arch bridges. We continue to use the notations $\sbq(\uu)$ and $\tbq(\uu)$ to denote the source and the target functions of this quiver.

Say that an arrow $\uu$ in $\bHQ_i(\xx_0)$ is an \emph{arch bridge} (resp. \emph{half $i$-arch bridge}, \emph{zero $i$-arch bridge}, \emph{reverse half arch bridge}) if $\uu$ is a weak arch bridge (resp. weak half $i$-arch bridge, weak zero $i$-arch bridge, weak reverse half arch bridge) and $\uu\neq\hh(\Pp)$ for any path $\Pp$ of length at least $2$ in $\overline{\Q}^{\mathrm{Ba}}_i(\xx_0)$ with $\sbq(\uu)=\sbq(\Pp)$ and $\tbq(\uu)=\tbq(\Pp)$.

The extended arch bridge quiver, denoted $\HQ_i(\xx_0)$, is the subquiver of $\bHQ_i(\xx_0)$ consisting of only arch bridges, zero $i$-arch bridges, half arch $i$-bridges and reverse half arch bridges.

\begin{rmk}
The quiver $\Q^{\mathrm{Ba}}_i(\xx_0)$ is a subquiver of $\HQ_i(\xx_0)$, and the latter is a subquiver of $\overline{\Q}^{\mathrm{Ba}}_i(\xx_0)$. \end{rmk}

\begin{example}
Figure \ref{archbrquivLambda'} shows the arch bridge quiver of the algebra $\Lambda'$ from Figure \ref{Nonassociative}. Note that the arrows $JA$ and $JiheD$ are not bridges.

Figure \ref{extarchbrquivex} shows the extended arch bridge quiver of the algebra $\Lambda''$ from Figure \ref{C-opposite}, where the reference point is chosen to be a length $0$-string at vertex $v_1$. Note that the arrow $IbA$ is an half $1$-arch bridge that is not a half bridge.
\begin{figure}[h]
\begin{minipage}[b]{0.4\linewidth}
\centering
\begin{tikzcd}
cbA \arrow[d, "bA"'] \arrow[r, "JA"]        & lK                    \\
biheD \arrow[r, "eD"'] \arrow[ru, "JiheD"'] & egF \arrow[u, "Jih"']
\end{tikzcd}
    \caption{A part of $\HQ$ for $\Lambda'$}
    \label{archbrquivLambda'}
\end{minipage}
\hspace{0.7cm}
\begin{minipage}[b]{0.53\linewidth}
\centering
\begin{tikzcd}
{1_{(v_1,i)}} \arrow[d, "edcbA"'] \arrow[r, "IbA"] & kJ                       \\
edF \arrow[r, "dF"]                                & dcbhG \arrow[u, "IbhG"']
\end{tikzcd}
    \caption{$\HQ_1(1_{(v_1,i)})$ for $\Lambda''$}
    \label{extarchbrquivex}
\end{minipage}
\end{figure}
\end{example}

The following two results extend Theorem \ref{weakarchbridgeHreduced} and its converse, Proposition \ref{Hreducedweakarchbridge}, respectively.
\begin{theorem}
Suppose $\uu$ is an arrow in $\bHQ_i(\xx_0)$. Then
\begin{itemize}
    \item if $\sbq(\uu)\neq\xx_0$ then $\uu$ is an H-reduced string;
    \item if $\sbq(\uu)=\xx_0$ then $\uu\xx_0$ is an H-reduced string relative to $(i,\xx_0)$.
\end{itemize}
\end{theorem}

\begin{proposition}\label{extHreducedweakarchbr}
Suppose $\vv,\vv'$ are distinct vertices of $\bHQ_i(\xx_0)$. If a string $\yy$ satisfies any of the following conditions:
\begin{itemize}
    \item[$(\vv\neq\xx_0)$] $\yy$ is an H-reduced string and $\vv'\yy\vv$ is a skeletal string;
    \item[$(\vv=\xx_0)$] $\yy\xx_0$ is an H-reduced string relative to $(i,\xx_0)$ and $\vv'\yy$ is a skeletal string,
\end{itemize}
then $\vv\xrightarrow{\yy}\vv'$ is an arrow in $\bHQ_i(\xx_0)$.
\end{proposition}

The following result extends Proposition \ref{pathforweakarchbridge}.
\begin{proposition}
Suppose $\uu$ is an arrow in $\bHQ_i(\xx_0)$. Then there is some path $\Pp$ in $\HQ_i(\xx_0)$ with $\sbq(\uu)=\sbq(\Pp)$ and $\tbq(\uu)=\tbq(\Pp)$ such that $\uu=\hh(\Pp)$.
\end{proposition}

The following result extends Theorem \ref{hisinjective}.
\begin{theorem}
Suppose $\Pp,\Pp'$ are paths in $\HQ_i(\xx_0)$ such that $\sbq(\Pp)=\sbq(\Pp')$, $\tbq(\Pp)=\tbq(\Pp')$ and $\hh(\Pp)=\hh(\Pp')$ then $\Pp=\Pp'$.
\end{theorem}

Finally the following result extends Theorem \ref{Hassociativity}.
\begin{theorem}
Suppose $\xx_0\xrightarrow{\uu_1}\bb_1\xrightarrow{\uu_2}\bb_2\xrightarrow{\uu_3}\vv$ is a sequence of arrows in $\bHQ_i(\xx_0)$, where $\uu_2,\uu_3$ are non-trivial. Then $\uu_3\ch^{(i,\xx_0)}(\uu_2\ch^{(i,\xx_0)}\uu_1)=(\uu_3\ch^{(i,\xx_0)}\uu_2)\ch^{(i,\xx_0)}\uu_1$.
\end{theorem}

\section{The extended semi-bridge quiver}\label{newextsemibrquiv}
In this section we identify a subquiver of the extended arch bridge quiver that properly contains the extended bridge quiver.

Motivated by the property in Remark \ref{exitisexitforarrow} we identify a subset of the set of weak arrows such that this desirable property is also true of the subset.

\begin{definition}
Say that a weak bridge (respectively, a torsion weak reverse half bridge) $\uu$ is a \emph{semi-bridge} (resp. \emph{torsion reverse semi-bridge}) if for any non-trivial factorization $\uu=\uu_2\circ\uu_1$ with $\uu_1$ a weak bridge we have $\beta(\uu)\neq\beta(\uu_1)$.
\end{definition}

A bridge (resp. torsion reverse half bridge) is trivially a semi-bridge (resp. torsion reverse semi-bridge) but the converse need not be true as the following example illustrates.
\begin{example}\label{semibrgexmple}
Consider the weak bridge $\uu:=\bb_1\xrightarrow{JA}\bb_4$ in the string algebra $\Lambda'$ from Figure \ref{Nonassociative}. It is readily verified that $\uu=(\uu_3\circ\uu_2)\circ\uu_1$ is the only factorization and thus $\uu$ is not a bridge. But since $\beta(\uu)\neq\beta(\uu_1)$, $\uu$ is a semi-bridge.
\end{example}

We get a description of abnormal semi-bridges from Corollary \ref{abnwbcancel}, Proposition \ref{characterizeabnormality} and Remark \ref{associativenaafact}.
\begin{corollary}\label{newabnormaluniquefactor}
An abnormal weak bridge is a semi-bridge and can be uniquely factorised as an associative composition of abnormal bridges.
\end{corollary}

\begin{example}\label{exnormalsemibr}
In contrast to the above corollary the semi-bridge from Examples \ref{semibrgexmple} and \ref{nonassociativityweakbridge} shows that the factorization of a normal semi-bridge need not be associative.
\end{example}

Now we introduce some notation to give a useful characterization of semi-bridges.

Given an exit syllable $\vv$ of a band $\bb$ let $\lambda^b(\vv)$ and $\lambda^r(\vv)$ denote the set of all semi-bridges with exit $\vv$ and the set of all maximal torsion reverse semi-bridges with exit syllable $\vv$ respectively. Then set $\lambda(\vv):=\lambda^b(\vv)\sqcup\lambda^r(\vv)$. Define a binary relation $\triangleleft$ on the set $\bar\lambda(\vv)$ by $\uu\triangleleft\uu'$ if $\uu'=\uu''\circ\uu$ for some weak bridge $\uu''$. The relation $\triangleleft$ is clearly transitive. Moreover since $\Lambda$ is domestic it is anti-symmetric.

The proof of the next result is obvious, and thus is omitted. 
\begin{proposition}\label{newsemibridgeminimal}
For a band $\bb$ and $\vv\in\mathcal E(\bb)$ the set $\lambda(\vv)$ of semi-bridges and maximal torsion reverse half bridges with exit $\vv$ is precisely the set of minimal elements of $(\bar\lambda(\vv),\triangleleft)$.
\end{proposition}

The next result follows readily from Proposition \ref{normalexitnormalexit}.
\begin{corollary}\label{semibridinitialabnfacts}
If $\uu$ is a semi-bridge or a torsion reverse weak half bridge and $\uu=\uu_2\circ\uu_1$ is a non-trivial factorization into weak bridges then $\uu_1$ is abnormal.
\end{corollary}

\begin{proposition}\label{newsemibrisarchbr}
If $\uu$ is a semi-bridge (resp. a torsion reverse semi-bridge) then $\uu$ is an arch bridge (resp. a torsion reverse arch bridge).
\end{proposition}

\begin{proof}
We only prove the result when $\uu$ is a semi-bridge; the other case has similar proof. In view of Remark \ref{bridgeisarchbridge}, we may assume that $\uu\notin\Q^{\mathrm{Ba}}$.

If $\uu\notin\HQ$ then $\uu=\uu_2\ch\uu_1$ for some weak bridges $\uu_1,\uu_2$. Since $\uu$ is band-free we also have $\uu=\uu_2\circ\uu_1$. Then Corollary \ref{semibridinitialabnfacts} gives that $\uu_1$ is abnormal. Since $\uu$ factors through $\uu_1$, Remark \ref{factgivesperp} gives $\beta(\uu)\perp\beta(\uu_1)$. 

Let $\uu_1^e=\xx_2\xx_1$ be a partition such that $\beta(\uu)\xx_1$ is a string. Let $\ww$ denote the left substring of $\sbq(\uu_1)$ such that $\uu_1^c\ww$ is a string. Since $\uu_1$ is abnormal, $\xx_2\xx_1\ww$ is a string but $\xx_2\bub{\uu_1}\xx_1\ww$ is not. Hence $\HRed{\tbq(\uu_1)}(\sk{\uu_2\tbq(\uu_1)\uu_1})$ does not exist, which is a contradiction to $\uu=\uu_2\ch\uu_1$.
\end{proof}

\begin{example}\label{ex:archnotsemi}
The converse of the above result is not true as the weak bridge $\uu:=IbAL:\bb'_1\to\bb'_4$ in $\overline\Q^{\mathrm{Ba}}(\Lambda'')$ from Example \ref{nonassociativityweakbridge} is an arch bridge but not a semi-bridge since it factors as $\uu'_3\circ(\uu'_2\circ\uu'_1)$, where $\beta(\uu)=\beta(\uu'_2\circ\uu'_1)$.
\end{example}

Motivated by Proposition \ref{newsemibridgeminimal} we introduce a similar notation for weak half bridges and weak zero bridges. 

We can define an anti-symmetric, transitive binary relation $\triangleleft$ on $\bar\lambda_i(\xx_0)$ in analogous manner. Let $\lambda_i(\xx_0)$ denote the set of minimal elements in $(\bar\lambda_i(\xx_0),\triangleleft)$. We have natural partitions $\lambda_i(\xx_0)=\lambda^h_i(\xx_0)\sqcup\lambda^z_i(\xx_0)$. Finally set $\lambda(\xx_0):=\lambda_1(\xx_0)\sqcup\lambda_{-1}(\xx_0)$ with natural partition $\lambda(\xx_0)=\lambda^h(\xx_0)\sqcup\lambda^z(\xx_0)$. The elements of $\lambda^h(\xx_0)$ and $\lambda^z(\xx_0)$ will be referred to as \emph{half semi-bridges} and maximal \emph{torsion zero semi-bridges} respectively.

\begin{proposition}\label{trianglefactabnhalfbridge}
If $\uu_1,\uu_2\in\bar\lambda_i^h(\xx_0)$ are abnormal and distinct then either $\uu_1\triangleleft\uu_2$ or $\uu_2\triangleleft\uu_1$.
\end{proposition}

\begin{proof}
Since $\uu_1\neq\uu_2$, in view of Proposition \ref{uniquebanduniqueabnormalhalfbridge}, we have $\tbq(\uu_1)\neq\tbq(\uu_2)$. Let $\gamma$ be the common first syllable of $\tbq(\uu_1)\uu_1$ and $\tbq(\uu_2)\uu_2$ with $\theta(\gamma)=i$. Since both $\uu_1,\uu_2$ are abnormal, $\gamma$ is a common syllable of $\tbq(\uu_1)$ and $\tbq(\uu_2)$. Without loss of generality, there is an abnormal weak bridge $\tbq(\uu_1)\xrightarrow{\uu}\tbq(\uu_2)$ with $\gamma\in\uu^c$. Since $|\uu^c|>0$, Proposition \ref{constantsign} guarantees $-\theta(\beta(\uu))=\delta(\uu^c)=\theta(\gamma)$. 

Moreover, there is a partition $\uu^c=\xx_2\xx_1$ such that $|\xx_2|>0$ and $\xx_2\xx_0$ is a string. Since $\delta(\xx_2)\neq0$, we see that $\xx_2$ is the maximal common left substring of $\tbq(\uu_1)\uu_1$ and $\tbq(\uu_2)\uu_2$.

If $\xx_2$ is the maximal common left substring of $\uu_1$ and $\uu_2$ then the concatenation $\uu\uu_1$ exists since the first syllable of $\uu$ and the last syllable of $\uu_1$ have opposite signs. If $\xx_2$ is not a common left substring of $\uu_1$ and $\uu_2$ then $\xx_2=\uu^c$ and $|\uu_1|=0$. The concatenation $\uu\uu_1$ exists in this case as well since $\uu$ is a left substring of $\bla{\uu}$ and $\tbq(\uu_2)\uu_2\xx_0$ is a string.

The first syllable of $\tbq(\uu_2)\uu_2$ from left that is not a syllable of $\bua{\uu_1}$ is $\beta(\uu)$ since $\xx_2$ forks. Moreover, since $\delta(\xx_2)=-\beta(\uu)$, the concatenation $\uu_2\bua{\uu_1}$ exists. Thus $\tbq(\uu_2)\uu_2\bua{\uu_1}=\tbq(\uu)\uu\uu_1$, which implies that $\uu_2\bua{\uu_1}=\uu\uu_1$. Thus we have showed that $\uu_2=\uu\circ\uu_1$, i.e., $\uu_1\triangleleft\uu_2$.
\end{proof}

We note an immediate consequence of the above.
\begin{corollary}\label{uniqueabnhalfbridge}
If $\uu\in\lambda^h_i(\xx_0)$ is abnormal then each $\uu'\in\lambda^h_i(\xx_0)$ distinct from $\uu$ is normal.
\end{corollary}

In view of the above corollary, let $\lambda^a_i(\xx_0)$ denote the unique abnormal half semi-bridge, if exists.

\begin{rmk}\label{zerolengthweakbridge}
If $\uu\in\bar\lambda_i(\xx_0)$ and $|\uu|=0$ then $\uu$ is an abnormal half bridge.
\end{rmk}

Existence of a zero length element in $\lambda_i(\xx_0)$ has strong consequences.
\begin{proposition}\label{lengthzerohalfbruniqueness}
If $\lambda^a_i(\xx_0)$ exists and $|\lambda^a_i(\xx_0)|=0$ then $\lambda_i(\xx_0)=\{\lambda^a_i(\xx_0)\}$.
\end{proposition}

\begin{proof}
Suppose $\uu\in\bar\lambda^h_i(\xx_0)$ and $\uu\neq\lambda^a_i(\xx_0)$. By definition the concatenation $\tbq(\uu)\uu\xx_0$ exists. Since the first and the last syllables of $\tbq(\lambda^a_i(\xx_0))$ are of opposite sign and $\theta(\tbq(\uu)\uu)=i=\theta(\tbq(\lambda^a_i(\xx_0))\lambda^a_i(\xx_0))$, the concatenation $\tbq(\uu)\uu\tbq(\lambda^a_i(\xx_0))\lambda^a_i(\xx_0)\xx_0$ also exists, and thus $\lambda^a_i(\xx_0)\triangleleft\uu$.

A similar argument holds if $\uu\in\bar\lambda^z_i(\xx_0)$.
\end{proof}

\begin{rmk}
If $\lambda^a_i(\xx_0)$ exists and $|\lambda^a_i(\xx_0)|=0$ then $i=1$.
\end{rmk}

Note that a half bridge (resp. maximal torsion zero bridge) is clearly a half semi-bridge (resp. maximal torsion zero semi-bridge). The converse however is not true.
\begin{rmk}\label{semihalfnotstrong}
Suppose $\uu=\uu'\circ\lambda^a_{-i}(\xx_0)$, $\beta(\uu')\xx_0$ is a string and $\uu'\in\lambda(\beta(\uu'))$ then $\uu\in\lambda_i(\xx_0)$.
\end{rmk}

The following results states that the cases described in the above remark are the only elements of $\lambda_i(\xx_0)$ that are neither half bridges nor torsion zero bridges.
\begin{theorem}\label{halfbridgeminimal}
If $\uu\in\lambda^h_i(\xx_0)$ (resp. $\uu\in\lambda^z_i(\xx_0)$) is not a half bridge (resp. a maximal torsion zero bridge) then there is a factorization $\uu=\uu'\circ\lambda_{-i}^a(\xx_0)$ such that $\uu'\in\lambda^b(\beta(\uu'))$ (resp. $\uu'\in\lambda^r(\beta(\uu'))$) and $\beta(\uu')\xx_0$ is a string.
\end{theorem}

\begin{proof}
Suppose $\uu\in\lambda^h_i(\xx_0)$ is not a half bridge. Then there is a factorization $\uu=\uu_2\circ\uu_1$ for some half bridge $\uu_1$. Since $\uu$ is $\triangleleft$-minimal in $\bar\lambda_i(\xx_0)$ we immediately conclude that $\uu_1\in\lambda^h_{-i}(\xx_0)$. In other words, $\theta(\tbq(\uu_1)\uu_1)=-i$.

If $\uu_1$ is normal then $\theta(\uu_1^o)=-i$ and $\uu_1^o$ is a substring of both $\uu_2\uu_1$ and $\uu=\uu_2\circ\uu_1$. Since $\theta(\uu)=i$, we arrive at a contradiction. Hence $\uu_1$ is abnormal. By Corollary \ref{uniqueabnhalfbridge} we conclude that $\uu_1=\lambda^a_{-i}(\xx_0)$.

Note that $|\uu|>0$ in view of Remark \ref{zerolengthweakbridge}. In order to ensure that $\theta(\uu_2\circ\uu_1)=i$ we have one of the following two cases. 
\begin{itemize}
    \item If $\uu_2\uu_1$ is band-free then $|\uu_1|=0$.
    \item If $\uu_2\uu_1$ is not band-free then $\bua{\uu_1}$ is a left substring of $\uu_2\uu_1$.
\end{itemize}

In both cases $\beta(\uu_2)\xx_0$ is a string. Finally, $\triangleleft$-minimality of $\uu_2$ in $\bar\lambda(\beta(\uu_2))$ is guaranteed by $\triangleleft$-minimality of $\uu$ in $\lambda_i(\xx_0)$.

If $\uu\in\lambda^z_i(\xx_0)$ is not a torsion zero bridge then the proof is similar.
\end{proof}

The following is an analogue of Proposition \ref{newsemibrisarchbr} which provides a subclass of half/zero $i$-arch bridges. 
\begin{corollary}
If $\uu\in\lambda^h_i(\xx_0)$ (resp. $\uu\in\lambda^z_i(\xx_0)$) then $\uu$ is a half $i$-arch bridge (resp. a zero $i$-arch bridge).
\end{corollary}

\begin{proof}
We only prove the result when $\uu\in\lambda^h_i(\xx_0)$; the other proof is similar.

Suppose $\uu$ is not a half $i$-arch bridge. Then $\uu=\uu_2\ch\uu_1$ for some $\uu_1\in\bar\lambda^h_i(\xx_0)$ and $\uu_2\in\overline\Q^{\mathrm{Ba}}$. Since $\uu$ is band-free we see that $\uu=\uu_2\circ\uu_1$. But then Theorem \ref{halfbridgeminimal} gives that $\uu_1=\lambda^a_{-i}(\xx_0)\notin\lambda^h_i(\xx_0)$, a contradiction.
\end{proof}

\vspace{0.2in}
\noindent{}Shantanu Sardar\\
Indian Institute of Technology Kanpur\\
Uttar Pradesh, India\\
Email: \texttt{shantanusardar17@gmail.com}

\vspace{0.2in}
\noindent{}Corresponding Author: Amit Kuber\\
Indian Institute of Technology Kanpur\\
Uttar Pradesh, India\\
Email: \texttt{askuber@iitk.ac.in}\\
Phone: (+91) 512 259 6721\\
Fax: (+91) 512 259 7500
\end{document}